\documentclass[12pt]{amsart}

\usepackage[utf8]{inputenc}
\usepackage[english]{babel}
\usepackage[all]{xy}
\usepackage{amssymb,latexsym,amsmath,amsthm, mathrsfs,faktor,stmaryrd}
\usepackage{dsfont,upquote}
\usepackage{booktabs}
\usepackage{graphicx,wrapfig,lipsum,float}
\usepackage{setspace}
\usepackage{fullpage}
\usepackage{color}
\usepackage{wallpaper}
\usepackage{chngcntr}
\usepackage{xcolor}
\usepackage{etoolbox}
\definecolor{darkblue}{rgb}{0,0,0.7}
\definecolor{darkgreen}{rgb}{0,0.5,0}

\usepackage{indentfirst}
\usepackage{hyperref}
\usepackage[shortlabels]{enumitem}
\usepackage{pdfpages}
\hypersetup{
     colorlinks=true,
     linkcolor=blue,
     filecolor=blue,
     citecolor = blue,      
     urlcolor=cyan,
     }
\usepackage{doi}

\newcommand{\bparagraph}[1]{\paragraph{#1}\mbox{}\\\indent}

\let\it\textit

\let\bf\textbf

\let\emptyset\varnothing

\let\bar\overline

\let\tilde\widetilde

\let\subset\subseteq

\let\epsilon\varepsilon
\let\longto\longrightarrow
\let\hookto\hookrightarrow

\let\d\partial
\def\ll{\left\llbracket}
\def\rr{\right\rrbracket}
\def\<{\left\langle}
\def\>{\right\rangle}

\docsvlist{Z,Q,R,C,RP,CP,FP,P,G,B,D}

\docsvlist{sign,Tan,sm,codim,dim,msdim,closure,Div,Var,rank,image,ker,exp,det,Sec}

\docsvlist{S,id,Hom,Gr,Bl,Fl,A}

\docsvlist{N,U,L,W,F,B,O,E,L,T,S,Q,V,L,F,H}

\def\NN{\mathsf{N}}

\def\TT{\mathbb{T}}
\def\GGr{{\mathbb{G}\mathrm{r}}}
\def\PGL{{\mathbb{P}\mathrm{GL}}}

\newtheoremstyle{breakdefn}
  {\topsep}{\topsep}%
  {\normalfont}{}%
  {\bfseries}{}%
  {\newline}{}%
\theoremstyle{breakdefn}
\newtheorem{defn}{Definition}[section]
\newtheorem{eg}[defn]{Example}
\newtheorem{rmk}[defn]{Remark}

\newtheoremstyle{break}
  {\topsep}{\topsep}%
  {\itshape}{}%
  {\bfseries}{}%
  {\newline}{}%
\theoremstyle{break}
\newtheorem{thm}[defn]{Theorem}
\newtheorem{lem}[defn]{Lemma}
\newtheorem{prop}[defn]{Proposition}
\newtheorem{cor}[defn]{Corollary}

\title{Algebraically Skew Embeddings of Curves}
\author{Andy B. Day}
\date{}
\email{andyday@psu.edu}

\begin{document}

\begin{abstract}
    Given a smooth complex variety $X$, an algebraically skew embedding of $X$ is an embedding of $X$ into a complex projective space $\P^N$ such that for any two points $x,y\in X$, their embedded tangent spaces in $\P^N$ do not intersect. In this work, we establish an upper bound and a lower bound of the minimal dimension $N=\msdim X$ such that there exists an algebraically skew embedding into $\P^N$ in terms of the dimension of the given smooth variety $X$. Then we further classify the algebraic curves in terms of their minimal skew embedding dimensions, and apply the same technique to other one-parameter family of lines.
\end{abstract}

\maketitle

\section{Introduction}

\bparagraph{\bf{History, motivations, and relevant works}}
This work is motivated by the skew embedding problems from differential topology. The questions of skew embedding started in a lecture given by H. Steinhaus in 1966. In the lecture, he asked: Are there smooth loops in the space without a pair of parallel tangents? He called such a curve a skew loop. The problem was soon solved by B. Segre in \cite{seg} in 1968. Segre used convex geometry to construct several skew loops in $\R^3$. Despite the existence of these skew loops, Segre noted that such loops lie on neither ellipsoids nor elliptic paraboloids. Later, Q-Y. Wu showed that there exist skew loops of every knot class \cite{Wu}, and B. Solomon further showed that skew loops of every knot class also exist in the flat $3$-torus \cite{Tori}. On the other hand, M. Ghomi and Solomon showed that Ellipsoids are the only closed surfaces that contain no skew loops \cite{Gho02}.

In \cite{White}, S. S. Chern and J. H. White started the multidimensional case of skew loops, called skew branes. In particular, they showed that codimension $2$ submanifolds in the unit sphere of Euclidean space cannot be a skew brane. S. Tabachnikov later generalized this result to codimension $2$ submanifolds in any quadratic hypersurfaces \cite{tab2003}. Parallelly, Chern and H-F. Lai discovered that the Euler characteristic of a codimension $2$ skew brane in an even dimension Euclidean space must vanish \cite{Lai}. Later, Tabachnikov and Y. Tyurina studied the number of pairs of parallel tangents of surfaces in $\R^4$ as well as codimension $2$ skew branes in an odd codimension Euclidean space by investigating their Gauss images \cite{Tab2010}.

Meanwhile, in \cite{Tab}, Ghomi and Tabachnikov asked: Given any manifold $M$, does there exist an embedding of $M$ to an euclidean space $\R^N$ such that for each pair of points in the image of $M$, their tangent spaces to $M$ neither intersect nor contain parallel lines? Ghomi and Tabachnikov called such an embedding a totally skew embedding. They further asked, if such embeddings exist, what is the minimal dimension of $\R^N$? In \cite{Tab}, Tabachnikov proved, among other things, that such an embedding always exists, and for any $n$-dimensional manifold $M$, the minimal dimension of $\R^N$ is always between $2n+1$ and $4n+1$. 

On the one hand, the condition of the totally skewness defined above is equivalent to saying that for each pair of points of the image of $M$, their tangent spaces do not intersect in the projectification of $\R^N$. On the other hand, we know that there is a rich theory in algebraic geometry on the intersection of subvarieties in a projective space, especially for spaces over an algebraically closed characteristic zero field. Thus, it is natural to study the skew embeddings of complex projective varieties in complex projective spaces. In this work, we take on the Gauss map approach for the complex projective subvarieties, and investigate other applications of the method we developed.

It turns out that the complex projective version of skew embeddings fits into the theory around Terracini loci nicely. The notion of Terracini loci was first introduced in \cite{Terr} by E. Ballico and L. Chiantini. Roughly speaking, the Terracini loci of a smooth projective subvariety $X\subset\P^N$ capture the finite subsets of $X$ such that the tangent spaces at the points in a given set are not in general position. The Terracini loci of Veronese embeddings have been studied in \cite{Terr} and \cite{TerrAmple}, and the ones of Segre embeddings have been studied in \cite{Terr3}. In \cite{TerrCurve} E. Ballico and L. Chiantini have looked into the Terracini loci of the curves in $\P^3$ and of the canonical curves.

C. Ciliberto kindly informed us that in his recent paper \cite{Ciro}, he studied the Terracini loci of curves in $\P^4$, which coincides with Lemma \ref{length} in the present paper (while Lemma \ref{length} counts the number of ``non-skew" ordered pairs of tangent lines instead of the length of the Terracini locus). Our approach to proving Lemma \ref{length} is based on an explicit blowup construction, which differs from the one in Ciliberto's work. In \cite{Ciro}, Ciliberto observed that, by pulling back the theta divisor to the incidence variety associated to the Terracini locus, one can apply the Grothendieck-Riemann-Roch theorem and the Porteous formula to compute the length of the Terracini loci. We believe our method suggest new directions different from those in the existing Terracini loci literature. For instance, our method can be applied to calculate the locus of intersection of families of linear projective subspaces, as described in Section \ref{scroll3} and Section \ref{scroll4}. More significantly, the difficulty encountered in the surface case discussed in \cite{Ciro} does not appear to arise with our approach. Thus, the present paper offers a different approach for studying the Terracini loci of higher-dimensional varieties.

Independently, the complex version of affine spacial skew curves has occurred in \cite{Zai99}.
\hfill\\

\bparagraph{\bf{Basic notations}}
In this paper, we denote $\P$ or $\P^N$ for a complex projective space. If $V$ is a vector space, then $\P(V)$ is the projectivization of $V$, so $\P^N=\P(\C^{N+1})$. Unless otherwise specified, all the dimensions of a variety refer to it's complex dimension. The Grassmannian of $n$ dimensional linear subspaces in an $N$-dimensional complex vector space $V$ is denoted by $\Gr(n,N)=\Gr(n,V)$, and the Grassmannian for $(n-1)$-dimensional linear planes in $\P(V)$ is denoted by $\GGr(n-1,N-1)=\GGr(n-1,\P(V))$. If no confusion could arise, we sometimes simply write $G$ for the Grassmannian in question.

For any subset $U\subset\P^N$, we denote its Zariski closure by $\bar U$. For any variety $X$, we denote its tangent sheaf (or tangent bundle if $X$ is smooth) by $\T_X$ and the tangent space at a smooth point $x\in X$ by $\T_{x,X}$. For a pair of smooth varieties $X\subset Y$, we write $\N_{X,Y}$ for the normal bundle of $X$ over $Y$ and $\N_{x,X,Y}$ for the fibre of $\N_{X,Y}$ at $x$. For any projective subvariety $X\subset\P$ and any smooth point $x\in X$, we write $T_x=T_{x,X}\subset\P$ for the embedded tangent space of $X$ at $x$. For any two planes (linear subspaces) $K,L\subset\P$, we use $\<K,L\>$ to denote the plane spanned by $K$ and $L$ in $\P$.

For any variety $X$ with pure dimension $n$, we denote the Chow ring of $X$ by $\A(X)$. The group of dimension $r$ classes in $\A(X)$ is denoted by $\A_r(X)$, and the group of codimension $s$ classes is denoted by $\A^s(X)$. Thus, $\A_r(X)=\A^{n-r}(X)$. Similarly, we denote the numerical groups of $X$ with codimension $s$ and dimension $r$ by $\NN^s(X)$, and $\NN_r(X)$. If $X$ is smooth, then we denote its numerical ring by $\NN(X)$. For any subscheme $Y\subset X$, the class of $Y$ in $\A(X)$ is denoted by $[Y]$. For any class $\alpha\in\A(X)$, its class in $\NN(X)$ is denoted by $\ll\alpha\rr$. For simplicity, we also write $\ll Y\rr$ for $\ll[Y]\rr$.
\hfill\\

\bparagraph{\bf{The algebraically skew embeddings}}
We define the skew embeddings of a projective variety as follows. 

\begin{defn}
Let $X\subset \P^N$ be a smooth projective variety. We say that $X\subset\P^N$ is skew if for any two distinct points $x,y\in X$, $T_{x,X}\cap T_{y,X}=\emptyset$.
\end{defn}

Since algebraic arguments often fail to tell the difference between two distinct points and a point with multiplicity two, we need a modified definition that excludes the ``infinitesimal nonskewness". For the case of curves, we use the following definition.

\begin{defn}
Let $X\subset \P$ be a smooth projective curve. We say that $X$ is algebraically skew if $X$ is sekw and the third osculating planes of $X$ are nowhere degenerate; that is, for any point $x_0\in X\subset\P(V)$ and any local lifting $f:\C\to V$ of $X$ with $[f(0)]=x_0$, the vectors $f(0),f'(0),f''(0),$ and $f'''(0)$ are linearly independent.
\end{defn}

For a general projective variety $X\subset\P$, there are several possible definitions for algebraically skewness. One such definition is to define the algebraically skewness of $X\subset\P^N=\P(V)$ by vanishing of the length $2$ Terracini locus $\TT(X,V,2)$, which consists of length $2$ zero dimensional subschemes $S\subset X$ satisfying $\dim V(-2S)>\dim V-2n-2$, c.f. \cite{TerrCurve}.

Another possible definition for algebraically skewness is to require that the projective second and third fundamental forms being nowhere vanishing. The nowhere vanishing condition of the second projective fundamental form is equivalent to say that the Gauss map from $X$ to the appropriate Grassmannian is a diffeomorphism to its image, and the nowhere vanishing condition of the third projective fundamental form would exclude the infinitesimal nonskewness. For more on the projective fundamental forms, see \cite{cartan}. 

Note that the notion of algebraically skew defined above can be viewed as a special type of higher ampleness. More precisely, recall that a line bundle $\L$ over $X$ is $k$-very ample if any $k+1$ points on $X$ span a $k$-plane in $X\hookto\P(H^0(L))$. In particular, if $\L$ is a $3$-very ample line bundle over a smooth projective variety $X$, then the embedded variety $X\hookto\P(H^0(L))$ is skew. For more on the notion of $k$-very ampleness, see \cite{kample}.

For any projective variety $X$, we denote the minimal algebraically skew embedding dimension of $X$ by
$$\msdim X:=\min\{N:\text{there exists an algebraically skew embedding }X\hookrightarrow\P^N\}.$$
The goal of this paper is to compute $\msdim X$ for curves by applying the excess intersection formula to the product of Gauss images of $X$ and a certain degenerate locus introduced in Section \ref{Drs} in the blowups of the product of Grassmannians.

One might consider solving this problem by applying the double point formula to the map from the projectivized tangent bundle $\P(\T_X\oplus\O_X(-1))$ to $\P^N$ with the tangent variety of $X$ being its image. However, since the tangent variety is always singular at $X$, one still needs to use the excess intersection approach to modify the double point formula, and we believe that working on $\P^N$ and its corresponding Grassmannian would be more convenient for future generalizations.
\hfill\\

\bparagraph{\bf{The structure of the paper}}
In Section \ref{bound}, we consider the dimension of the tangent variety of the secant variety of a projective variety and use it to deduce that $\msdim X\le4n+1$ for any $n$-dimensional variety $X$. Then we introduce the degeneracy loci $\D_r$ of pairs of linear subspaces in $\P$ with intersecting dimension $\ge r$ in $\GGr(n,N)\times\GGr(n,N)$. By considering the dimension of the intersection between $D_1=\D_0$ and $\Gamma:=C\times C\subset\GGr(n,N)\times\GGr(n,N)$, the product of the Gauss image $C:=\gamma(X)$ of $X$, we deduce that $3n\le \msdim X$. Note that the upper bound $4n+1$ agrees with the real smooth case, but the lower bound $3n$ is better than the real smooth case, which is $2n+1$.

Although the intersection $D_1\cap\Gamma$ we study in Section \ref{bound} gives us a lower bound of $\msdim X$, it is not enough to solve $\msdim X$ for a specific variety $X$ because $D_1$ is singular on the diagonal of $G\times G$. We resolve the singularity of $D_1$ for the case of curves in Section \ref{three} by blowing up the diagonal of $G\times G$. Then we compute the intersection between the proper transforms $\tilde D_1$ and $\tilde\Gamma$ of $D_1$ and $\Gamma$ for the case of $X\subset\P^3$. By comparing the product $[\tilde D_1][\tilde\Gamma]$ and the class of the diagonal component $\tilde\Delta_\Gamma$ of the intersection in the Chow ring of the blowup $B:=\Bl_{\Delta_G}(G\times G)$, we conclude that the twisted cubic is the only skew curve in $\P^3$. We also apply the same method to other one-parameter families $C\subset G$ and deduce that for a scroll to be skew, we must have $C\cong\P^1$.

In Section \ref{four}, we study the case of one-parameter families of lines in $\P^4$. By dimension counting we immediately see that $\tilde D_1$ and $\tilde\Gamma$ have an excess intersection component at the diagonal component $\tilde\Delta_\Gamma$. While the same method in section \ref{three} works for a generic one-parameter family of lines, we discover that $\tilde D_1$ and $\tilde\Gamma$ tangent along $\tilde\Delta_\Gamma$ when $\Gamma$ is the product of Gauss images. We resolve this tangency by blowing up $B:=\Bl_{\Delta_G}(G\times G)$ along $\tilde\Delta_\Gamma$. Then we consider the intersection class between the proper transforms $D_1^\dagger$ and $\Gamma^\dagger$ of $\tilde D_1$ and $\tilde\Gamma$ as well as the excess intersection class $\epsilon_{\Delta^\dagger_\Gamma}$ of the component on the diagonal. The main obstruction of computing the excess intersection class in question is the Chow ring of $\tilde D_1$ and the class of $\tilde\Delta_\Gamma$ in it. We establish the structure of $\A(\tilde D_1)$ by introducing an intermediate incidence variety $\hat D_1\subset\P\times G\times G$ between $\tilde D_1$ and $D_1$.  We discover that $\hat D_1$ can be treated as a biprojective bundle $\P(\Q)\times_{\P^4}\P(\Q)$. In addition, we interpret $\tilde D_1$ as the blowup of $\hat D_1$ along the fibrewise diagonal $\hat\Delta_D$. We establish the structure $\A(\tilde D_1)$ by realizing that $\hat\Delta_D$ is isomorphic a flag variety and the normal bundle of $\hat\Delta_D$ in $\hat D_1$ is isomorphic to a subbundle of $\T_{\hat\Delta_D}$. With these tools, we manage to compute the excess intersection term. Then we assert that a curve $X\subset\P^4$ is algebraically skew if and only if the intersection class $[D_1^\dagger][\Gamma^\dagger]$ agrees with the excess intersection class $\epsilon_{\Delta^\dagger_\Gamma}$ from the diagonal and conclude that the only algebraically skew curves in $\P^4$ are the rational normal curve and the elliptic normal curves.
\hfill\\

\bparagraph{\bf{Acknowledgment}}
This project is partially supported by the NSF Award DMS-2142966. The author is a graduate student at Pennsylvania State University. The author is grateful to his advisors Sergei Tabachnikov and John Lesieutre for the weekly meetings in writing this paper. The author thanks Jack Huizenga for several crucial conversations and advice. The author is grateful to Edoardo Ballico for pointing out the relation between skew embeddings and Terracini loci and for Ciro Ciliberto for pointing out his work on the case of Terracini loci of curves in $\P^4$. The author also benefits from several insightful email exchanges with Joseph Landsberg, Serge Lvovski, and Igor Dolgachev. The author also thanks Brendan Hassett for pointing out the potential double point formula approach to the problem. The author also thanks Dawei Chen for a discussion of whether the intersection component on the blown-up diagonal in the present paper is a ribbon.

\section{The basic bounds of $\msdim X$}\label{bound}

In this section, we explain the existence of skew embeddings of a projective variety $X$ by observing the skewness of Veronese embeddings of degree $\ge3$. We establish the general upper bound of the minimal complex algebraically skew embedding dimension $\msdim X$ by investigating the tangent of the secant variety. Then we introduce the degeneracy filtration of $\Gr(n,N)\times\Gr(n,N)$, and use them to compute a lower bound of $\msdim X$.

\subsection{Skewness of higher degree Veronese embeddings}\hfill\\

We first recall that rational curves of degree $\ge3$ are algebraically skew. This follows from \cite{Terr} or \cite{TerrCurve}. One can also prove this with straightforward computation.
\begin{lem}\label{racurve}
The rational curve given by
    \begin{align*}
    \tilde\iota_d:(s:t)\longmapsto(s^d:s^{d-1}t:\cdots:st^{d-1}:t^d)
\end{align*}
is skew for $d\ge3$.
\end{lem}

By the observation above, one can prove that the following projections from degree $\ge3$ Veronese embeddings are skew.

\begin{lem}\label{Ver}
For any $n\ge1$, the embedding of $\P^n$ into $\P^{\frac{1}{2}(n+1)n(d-1)+n}$ given by
$$\iota_d:(x_0:x_1:\cdots:x_n)\longmapsto(\cdots:x_i^kx_j^{d-k}:\cdots)$$
is skew, where $i,j$ run over all $0\le i<j\le n$ and $k=0,\cdots,d$ but without repeating the terms of the form $x_i^d$.
\end{lem}

\begin{proof}
The map $\iota_d$ is clearly injective and smooth. To prove the skewness, pick any two $x,y\in\P^n$, let $M$ be the image of $\iota_d$, and let $\tilde\iota_d:\C^{n+1}\to \C^{\frac{1}{2}(n+1)n(d-1)+n+1}$ be the lifting of $\iota_d$ given by the same formulation. Suppose that $T_{\iota_d(x),M}\cap T_{\iota_d(y),M}\ne\emptyset$. Then there exist $\alpha,\beta,a_0,\cdots,a_n,b_0,\cdots,b_n\in\C$, not all equal to zero, and $\tilde x,\tilde y\in\C^{n+1}$ in the cones of $x$ and $y$ respectively such that
\begin{align}\label{vero}
\alpha \tilde\iota_d(\tilde x)+\sum_{k=0}^na_k\frac{\d\tilde\iota_d}{\d x_k}(\tilde x)=\beta\tilde\iota_d(\tilde y)+\sum_{k=0}^nb_k\frac{\d\tilde\iota_d}{\d x_k}(\tilde y).
\end{align}
For any $i<j$, compare the entries in (\ref{vero}) corresponding to the terms $x_i^kx_j^{d-k}$ in the image of $\tilde\iota_d$ for $k=0,\cdots,d$, we obtain a system of equations that is equivalent to
$$T_{\iota_d(x_i:x_j),\iota_d(\P^1)}\cap T_{\iota_d(y_i:y_j),\iota_d(\P^1)}\ne\emptyset.$$
This contradicts to Lemma \ref{racurve} unless $(x_i:x_j)=(y_i:y_j)$. Since $i,j$ are arbitrary, we have $x=y$.
\end{proof}

Notice that the map in Lemma \ref{Ver} is the composition of a projection and the degree $d$ Veronese embedding, thus the Veronese embeddings of degree $\ge3$ are skew. In particular, $\iota_3:\P^2\hookto\P^8$ gives us $\msdim(\P^2)\le8<4\cdot2+1$.

In addition, observe that if there exists a skew embedding $X\hookrightarrow\P^N$ and $Y\hookrightarrow X$ is an embedding, then the composition $Y\hookrightarrow X\hookrightarrow\P^N$ is also skew. Thus, Lemma \ref{Ver} also implies that every projective variety admits a skew embedding.

\subsection{The basic upper bound of $\msdim X$}\hfill\\

To find the basic upper bound of $\msdim X$, suppose that $X\hookto\P^N$ is a skew embedding with $N=\msdim X$. Let $\Sec X\subset \P^N$ be the secant variety of $X$. By the Terracini's Lemma \cite{Terracini}, see also \cite[Proposition II.1.10]{Zak}, we have
\begin{align*}
\<T_{x,X},T_{y,X}\>=T_{z,\Sec X}
\end{align*}
for any generic smooth points $x,y\in X$ and any smooth point $z\in\<x,y\>\subset \Sec X$. This implies
$$\closure\left(\bigcup_{(x,y)\in (X\times X)}\<T_{x,X},T_{y,X}\>\right)=\Tan\Sec X,$$
where $\Tan\Sec X$ denotes the tangent variety of $\Sec X$.

Since the left-hand side of the above equation is the closure of a bundle over $(X\times X)\setminus\Delta_X$ with fibres of dimension $2n+1$, we have $\dim(\Tan\Sec X)=4n+1$.

Observe that if $L\subset\P^N$ is a $k$-plane with $L\cap \Tan\Sec X=\emptyset$, then the projection $\pi_L:X\to\P^{N-k-1}$ remains a skew embedding as $\dim\pi_L(\<T_{x,X},T_{y,X}\>)=2n+1$. Thus by choosing $\pi_L$ with $\dim L=N-4n-2$ and $L\cap \Tan\Sec X=\emptyset$, we skewly project $X$ to $\P^{4n+1}$ and deduce
$$\msdim X\le4n+1.$$
Note that the upper bound above can be deduced from \cite{Terr} as well.

\subsection{The degeneracy filtration of $\Gr(n,N)\times\Gr(n,N)$ and the lower bound of $\msdim X$}\label{Drs}\hfill\\
In this section, we introduce and study the degeneracy filtration $D_r(n,N)$'s of $\Gr(n,N)\times\Gr(n,N)$, which is a series of varieties that help to measure how much a projective subvariety fails to be skew. In particular, we obtain that $3n\le \msdim X$ in Corollary \ref{3n}, which is a lower bound stronger than the real smooth case.

\begin{defn}
    Let $V:=\C^N$. For any $N\ge n\ge1$, define
    $$D_r=D_r(n,N)=D_r(n,V):=\{(U,W):\dim(U\cap W)\ge r\}\subset\Gr(n,V)\times\Gr(n,V).$$
    We also write $\D_{r-1}=\D_{r-1}(n-1,N-1)=\D_{r-1}(\P(V))$ for the subvariety in $\GGr(n-1,\P(V))\times\GGr(n-1,\P(V))$ corresponding to $D_r(n,N)$.
\end{defn}

\begin{prop}\label{bipro}
    Let $D_r=D_r(n,N)$ be defined as above and $m:=\max\{2n-N,0\}$.
    \begin{enumerate}[(i)]
        \item We have the filtration
        $$D_m\supsetneq D_{m+1}\supsetneq\cdots\supsetneq D_n.$$
    
        \item $\dim D_r=r(N-r)+2(n-r)(N-n)$ for any $m\le r\le n$.
        
        \item For the base case $r=m$, we have $$D_m=\Gr(n,N)\times\Gr(n,N).$$
    \end{enumerate}
\end{prop}

\begin{proof}
$(i)$ and $(iii)$ are trivial. For $(ii)$, consider the projection $p: D_r(r,V)\to\Gr(r,V)$ defined by $p:(U,W)\mapsto U\cap W$. $p$ is clearly surjective. To count the dimension of fibres of $p$, fix any element $L\in\Gr(r,N)$ and observe that the map $(U,W)\mapsto(U/L,W/L)$ from $p^{-1}(L)$ to $D_0(n-r,V/L)$ is bijective. Thus,
\begin{align*}
    \dim D_r(n,N)=&\dim\Gr(r,V)+\dim D_0(n-r,V/L)\\
    =&\dim\Gr(r,V)+\dim(\Gr(n-r,V/L)\times\Gr(n-r,V/L))\\
    =&r(N-r)+2(n-r)(N-n).
\end{align*}
\end{proof}

\begin{cor}
    Let $C\subset\GGr(n,N)$ be an $m$-parameter family of $n$-dimensional linear subspaces in $\P^N$. Suppose that $C$ is skew, that is, $P_1\cap P_2=\emptyset$ for any $[P_1]\neq[P_2]\in C$, then $2n+m\le N$.
\end{cor}

\begin{proof}
    Set $\Gamma:=C\times C\subset G\times G:=\GGr(n,N)\times\GGr(n,N)$.
    
    Since, as subvarieties in $G\times G$, we have $\codim \Gamma=2(n+1)(N-n)-2m$ and $\codim \D_0(n,\P^N)=N-2n$, the \it{expected dimension} of the intersection $\Gamma\cap \D_0(n,\P^N)$ is
    \begin{align*}
        \dim(\GGr(n,N)\times\GGr(n,N))-\codim \Gamma-\codim \D_0(n,\P^N)=2n+2m-N.
    \end{align*}     
    
    Observe that $C$ is skew if and only if, set theoretically, $\D_0(n,\P^N)\cap\Gamma=\Delta_\Gamma$, where $\Delta_\Gamma$ is the diagonal of $\Gamma:=C\times C$. Since $\dim\Delta_\Gamma=m$, we must have $m\ge 2n+2m-N$ for $C$ to be skew. Thus $2n+m\le N$.
\end{proof}

In particular, we have the following lower bound for $\msdim X$.
\begin{cor}\label{3n}
    Let $X$ be a smooth projective variety of dimension $n$, then $\msdim X\ge3n$.
\end{cor}

Now, we compute the Schubert class of $D_1$ in $\A(\Gr(n,N)\times\Gr(n,N))$. Recall that by the standard Schubert calculus the Chow ring $\A(\Gr(n,N)\times\Gr(n,N))$ has basis $\sigma_a\otimes\sigma_b$'s, where $\sigma_a$'s and $\sigma_b$'s are Schubert cycles in $\A(\Gr(n,N))$; $a=(a_1,\cdots,a_n)$ and $b=(b_1,\cdots,b_n)$ are partitions in decreasing order that consist of integers between $N-n$ and $0$. 

\begin{lem}\label{cDr}
    The Schubert class of $D_1(n,N)$ is given by
    $$[D_1]=\sum_{i=0}^{N-2n+1}\sigma_i\otimes\sigma_{N-2n+1-i}.$$
\end{lem}

\begin{proof}
    Consider the projections $\pi_1,\pi_2:\Gr(n,N)\times\Gr(n,N)\to\Gr(n,N)$. Let $\S$ and $\Q$ be the universal sub and quotient bundles over $\Gr(n,N)$ respectively, and let $\iota:\S\to\V$ be the obvious inclusion from $\S$ to the rank $N$ trivial bundle $\V$. Denote $\Q_i:=\pi_i^*\Q$, $\S_i:=\pi_i^*\S$, $\V_i:=\pi_i^*\V$ and $\iota_i:=\pi_i^*\iota$ for $i=1,2$. Write $p_i:\V_1\oplus\V_2\to\V_i$ for the fibrewise projections. Identify $\V_1$ and $\V_2$ in the canonical way and denote the identified bundle by $\V$ again. Then we obtain $p_1,p_2:\V_1\oplus\V_2\to\V$, where $\V$ is now the rank $N$ trivial bundle over $\Gr(n,N)\times\Gr(n,N)$.

    Now, consider the diagram
    \[
    \xymatrixcolsep{1pc}
    \xymatrix{
    \S_1\oplus\S_2\ar[rr]^\phi\ar[rd]&
    &\V\ar[ld]\\
    &\Gr(n,N)\times\Gr(n,N)&
    }
    \]
    where $\phi:=(p_1\circ\iota_1)+(p_2\circ\iota_2)$.The locus $D_k(n,N)$ agrees with the $(2n-k)^{\text{th}}$ degeneracy locus of $\phi$. That is,
    $$D_k(n,N)=\{(L,K):\rank(\phi(L,K))\le2n-k\}.$$

    By \cite[Theorem 14.4]{fulton}, we then obtain $[D_1]=c_{N-2n+1}(\V-(\S_1\oplus\S_2))$. Recall the sequence $0\to\S_i\to\V\to\Q_i\to0$ and $c(\Q)=1+\sigma_1+\sigma_2+\cdots+\sigma_{N-n}$. Thus
    \begin{align*}
        [D_1]=&c_{N-2n+1}(\V-(\S_1\oplus\S_2))
        =\left\{\frac{c(\V)}{c(\S_1)c(\S_2)}\right\}_{N-2n+1}\\
        =&\left\{c(\V)c(\Q_1)c(\Q_2)\right\}_{N-2n+1}\\
        =&\left\{(1\otimes1)(c(\Q)\otimes1)(1\otimes c(\Q))\right\}_{N-2n+1}\\
        =&\sum_{i=0}^{N-2n+1}\sigma_i\otimes\sigma_{N-2n+1-i}.
    \end{align*}
\end{proof}

\section{The blowup of the Grassmannian product and the skew curve in $\P^3$}\label{three}

From the discussion in the previous section, we know that a variety $X\subset\P^N$ is skew if and only if $(\Gamma\cap\D_0)\setminus\Delta_G=\emptyset$, where $\Gamma$ and $\D_0$ are defined as in Corollary \ref{3n} and $\Delta_G$ is the diagonal of $\GGr(n,N)\times\GGr(n,N)$. To handle the intersection of $\Gamma$ and $ D_1$ along $\Delta_G$, we consider the proper transforms of $\D_0$ and $\Gamma$ in the blowup $B:=\Bl_{\Delta_G}\GGr(n,N)\times\GGr(n,N)$.

In this section, we study the structure of the Chow ring $\A(B)$ and compute the proper transforms of $D_1$ and $\Gamma$ for the case of one-parameter family of lines in $\P^N$. Then we apply these tools to study the skewness of one-parameter family of lines and the skew curve in $\P^3$.

\subsection{The structure of $\A(B)$}\hfill\\

Let $\pi:B\to G\times G$ be the blowdown map and $E:=\pi^{-1}(\Delta_G)$ be the exceptional divisor, then we have the following blowup diagram.
\begin{align}\label{blowup}
    \xymatrixcolsep{1.5pc}
    \xymatrix{
    E\ar@{^{(}->}[r]^-{j}\ar[d]_{\pi_E}
    &B\ar[d]^{\pi}\\
    \Delta_G\ar@{^{(}->}[r]^-{i}&G\times G
    }
\end{align}

Recall that the Chow ring $\A(B)$ is generated by $j_*(\A(E))$ and $\pi^*(\A(G\times G))$; also recall that the Chow ring of $E\cong\P(\N_{\Delta_G,G\times G})$ is generated by $\pi_E^*(\A(\Delta_G))$ and the class $\zeta$ of hyperplane sections. To apply the general theory of the multiplicative structure of blowups, see \cite[Section 13.6]{3264} or \cite{fulton}, to the case of $\A(B)$, we need to compute the pushforward and the pullback of $i$.

For convenience, we will abuse the notation and denote the generators of $\A(\Delta_G)\cong\A(G)$ by $\sigma_p$'s and the generators of $\A(G\times G)$ by $\sigma_a\otimes\sigma_b$. That is, the element denoted by $\sigma_p$ in $\A(\Delta_G)$ is the pullback of $1\otimes\sigma_p$ and $\sigma_p\otimes1$ via $i$. For any Schubert cycle $\sigma_a$, we write $|a|$ for the codimension of $\sigma_a$ and denote $\sigma_{\bar a}$ to be the Schubert cycle such that $\sigma_a\sigma_{\bar a}$ is the class of a point.

\begin{lem}\label{pupu}
    Let $i$ be the inclusion in the diagram (\ref{blowup}). Then
    \begin{enumerate}[(i)]
        \item $i^*(\sigma_a\otimes\sigma_b)=\sigma_a\sigma_b$;
        \item $i_*(\sigma_p)=\sum\sigma_a\otimes\sigma_b$, where the summation runs over all $\sigma_a\otimes\sigma_b\in\A_{|\bar p|}( G\times G)$ such that $\sigma_{\bar a}\sigma_{\bar b}=\sigma_{\bar p}$.
    \end{enumerate}
\end{lem}

\begin{proof}
    The first statement follows from $i^*(\sigma_a\otimes1)=\sigma_a$, $i^*(1\otimes\sigma_b)=\sigma_b$, and $i^*$ is a ring homomorphism.

    For the pushforward, fix any $\sigma_p\in\A(\Delta_G)$. Since $i$ is an injection, we have $i_*(\sigma_p)\in\A_{|p|}( G\times G)$. Thus $i_*(\sigma_p)=\sum_{|\bar a|+|\bar b|=|\bar p|}n_{ab}\sigma_a\otimes\sigma_b$ for some integers $n_{ab}$. By the push-pull formula, for any $\sigma_a\otimes\sigma_b\in\A(G\times G)$, we have
    $$(\sigma_{\bar a}\otimes\sigma_{\bar b})i_*(\sigma_p)=i_*(i^*(\sigma_{\bar a}\otimes\sigma_{\bar b})\sigma_p)=i_*(\sigma_{\bar a}\sigma_{\bar b}\sigma_p)=\left\{\begin{array}{ll}
        \text{the class of a point} & \text{if }\sigma_{\bar a}\sigma_{\bar b}=\sigma_{\bar p} \\
        0 & \text{otherwise.}
    \end{array}\right.$$
    Thus $n_{ab}=1$ if $\sigma_{\bar a}\sigma_{\bar b}=\sigma_{\bar p}$ and $n_{ab}=0$ for any other pairs $(a,b)$.
\end{proof}

With the observation above and the general theory of blowups, we obtain the following lemmas.

\begin{lem}\label{blring}
    Denote the pullbacks of Schubert cycles under $\pi_E$ in $\A(E)$ by $\bar{\sigma_p}$, and the class of hyperplane section in $\A(E)$ by $\zeta$. Then the multiplicative structure of $\A(B)$ is given by
    \begin{enumerate}[(i)]
        \item $j_*(\bar\sigma_p\zeta^r)j_*(\bar\sigma_q\zeta^s)=-j_*(\bar\sigma_p\bar\sigma_q\zeta^{r+s+1})$,
        \item $j_*(\bar\sigma_p\zeta^r)\pi^*(\sigma_a\otimes\sigma_b)=j_*(\bar\sigma_p\bar\sigma_a\bar\sigma_b\zeta^r)$;
        \item $\pi^*(\sigma_a\otimes\sigma_b)\pi^*(\sigma_c\otimes\sigma_d)=\pi^*((\sigma_a\sigma_c)\otimes(\sigma_b\sigma_d))$.
    \end{enumerate}
    The ring structure of $\A(B)$ is then given by the multiplicative rules above, together with the following relation
    \begin{enumerate}[(i)]
    \setcounter{enumi}{3}
        \item $\pi^*(i_*(\sigma_p))=j_*\left(\bar\sigma_p\sum_{i=0}^{\dim G-1}\pi_E^*c_i(\N_{\Delta_G, G\times G})\zeta^{\dim G-1-i}\right)$ for any $\sigma_p\in\A(\Delta_G)$.
    \end{enumerate}
\end{lem}

\begin{lem}\label{projbun}
    The ring structure of $\A(E)\cong\A(\P(\N_{\Delta_G, G\times G}))$ is given by
    $$\A(E)=\frac{\pi_E^*\A(\Delta_G)[\zeta]}{\sum_{i=0}^{(N-n)n}\pi_E^*c_i(\N_{\Delta_G, G\times G})\zeta^{(N-n)n-i}}$$
    where $\pi_E^*\A(G)[\zeta]$ is the polynomial ring with coefficients in $\pi_E^*\A(\Delta_G)$ and the hyperplane section class $\zeta$ as the variable, and $c_i(\N_{\Delta_G, G\times G})$'s are the Chern classes of the bundle $\N_{\Delta_G, G\times G}$ over $\Delta_G$.
\end{lem}

To understand $\P(\N_{\Delta_G, G\times G})$, we also need the following fact.

\begin{lem}\label{nortan}
    Let $Y$ be any smooth variety and $\iota:Y\hookto Y\times Y$ be the inclusion map from $Y$ to the diagonal $\Delta_Y$ of $Y\times Y$. Then there is a canonical isomorphism $\T_Y\to\iota^*\N_{\Delta_Y,Y\times Y}$, given by
    $$(y,v)\longmapsto\iota^{-1}((y,y),[0,v]),$$
    where $[0,v]$ is the normal vector in $\N_{\Delta_Y,Y\times T}:=\T_{Y\times Y}|_{\Delta_Y}/\T_{\Delta_Y}$ represented by the vector $(0,v)$.

    From now on, if no confusion could arise, we will identify the projective bundle $E:=\P(\N_{\Delta_G}, G\times G)\to\Delta_G$ and $\P(\T_G)\to G$ via the identification above.
\end{lem}

\subsection{The class of $\tilde D_1$}\hfill\\
Now we focus on the class $[\tilde{D_1}]=[\tilde{D_1}(n,N)]\in\A(B)$.

\begin{lem}\label{tilD1}
    The class of $[\tilde{D_1}(2,4)]$ and $[\tilde{D_1}(2,5)]$ are given by
    \begin{enumerate}[(i)]
        \item $[\tilde{D_1}(2,4)]=\pi^*(\sigma_1\otimes\sigma_0+\sigma_0\otimes\sigma_1)-j_*(2)$;
        \item $[\tilde{D_1}(2,5)]=\pi^*(\sigma_2\otimes\sigma_0+\sigma_1\otimes\sigma_1+\sigma_0\otimes\sigma_2)-j_*(5\bar\sigma_1+3\zeta)$
    \end{enumerate}
\end{lem}

\begin{proof}
    To begin, we know that the class $[\tilde{D_1}]$ is of the form $\pi^*([D_1])+j_*(\mu)$ for some $\mu\in\A_{\dim D_1}(E)=\A^{\codim_{ G\times G} D_1-1}(E)$. By Proposition \ref{cDr}, we know that $[D_1]=\sum_{i=0}^{N-2n+1}\sigma_i\otimes\sigma_{N-2n+1-i}$. By Proposition \ref{bipro}, we have $\codim D_1=N-2n+1$. Thus, we see that
    $$[\tilde{D_1}]=\pi^*\left(\sum_{i=0}^{N-2n+1}\sigma_i\otimes\sigma_{N-2n+1-i}\right)+\sum_{|a|\le N-2n}n_aj_*\left(\bar\sigma_a\zeta^{N-2n-|a|}\right)$$
    for some integers $n_a$'s.

    By Lemma \ref{blring}, we have
    \begin{align*}
        [\tilde D_1(n,N)][E]
        =&\left(\pi^*\left(\sum_{i=0}^{N-2n+1}\sigma_i\otimes\sigma_{N-2n+1-i}\right)+\sum_{|a|\le N-2n}n_aj_*\left(\bar\sigma_a\zeta^{N-2n-|a|}\right)\right)j_*(1)\\
        =&j_*\left(\sum_{i=0}^{N-2n+1}\bar\sigma_i\bar\sigma_{N-2n+1-i}-\sum_{|a|\le N-2n}n_a\bar\sigma_a\zeta^{N-2n+1-|a|}\right)\\
        =&j_*\left(\sum_{i=0}^{\lfloor\frac{1}{2}(N-2n+1)\rfloor}(N-2n+2-2i)\bar\sigma_{N-2n+1-i,i}-\sum_{|a|\le N-2n}n_a\bar\sigma_a\zeta^{N-2n+1-|a|}\right).
    \end{align*}
    Thus,
    \begin{align}\label{alg24}
        [\tilde D_1(2,4)][E]=j_*(2\bar\sigma_1-n_0\zeta)
    \end{align}
    and
    \begin{align}\label{alg25}
        [\tilde D_1(2,5)][E]=j_*(3\bar\sigma_2+\sigma_{11}-n_0\zeta^2-n_1\bar\sigma_1\zeta).
    \end{align}

    On the other hand, the class of $\tilde D_1\cap E$ can be computed as a degeneracy locus as follows. By Lemma \ref{nortan} we see that $E\cong\P(\T_G)\cong\P(\Hom(\S,\Q))$, where $\S$ and $\Q$ are the universal sub and quotient bundles of $\Gr(n,N)$ respectively. Let $\O_E(-1)\subset \pi_E^*\Hom(\S,\Q)\cong\Hom(\pi_E^*\S,\pi_E^*\Q)$ be the universal line bundle of the projective bundle $\P(\Hom(\S,\Q))$. Then we have the diagram

    \begin{align*}
        \xymatrixcolsep{2pc}
        \xymatrix{
        \O_E(-1)\otimes\pi_E^*\S\ar[rr]^-{\phi}\ar[rd]
        &&\pi_E^*\Q\ar[ld]\\
        &E&
        }
    \end{align*}
    where $\phi(v\otimes s):=v(s)$. Note that $\rank(\O_E(-1)\otimes\pi_E^*\S)=n$, $\rank\pi_E^*\Q=N-n$, and $\phi$ is a linear map between these two bundles.

    Observe that, under the identification $E\cong\P(\Hom(\S,\Q))$, we have
    \begin{align}\label{nortand1}
        \tilde D_1\cap E=\left\{(x,[v]):\rank v<n\right\}=\left\{(x,[v]):\rank \phi_{(x,[v])}\le n-1\right\},
    \end{align}
    and
    $$\codim_E \tilde D_1\cap E=N-2n+1=(\rank(\O_E(-1)\otimes\pi_E^*\S)-(n-1))(\rank \pi_E^*\S-(n-1)).$$
    Thus, by the Porteous formula (see \cite[Theorem 12.4]{3264} or \cite[Theorem 14.4]{fulton}), we have
    $$\left[\tilde D_1\cap E\right]=\left\{\frac{c(\pi_E^*\Q)}{c(\O_E(-1)\otimes\pi_E^*\S)}\right\}_{N-2n+1}=\left\{c(\pi_E^*\Q)s(\O_E(-1)\otimes\pi_E^*\S)\right\}_{N-2n+1}$$
    as a class in $\A(E)$, where $s(\O_E(-1)\otimes\pi_E^*\S)$ stands for the Segre class of $\O_E(-1)\otimes\pi_E^*\S$.
    
    Recall that $c(\pi_E^*Q)=1+\bar\sigma_1+\bar\sigma_2+\cdots+\bar\sigma_{N-n},c(\pi_E^*S)=1-\bar\sigma_1+\bar\sigma_{1^2}+\cdots+(-1)^n\bar\sigma_{1^n}$, and $c(\O_E(-1))=1-\zeta\in\A(E)$, where $\sigma_{1^k}:=\sigma_{1,1,\cdots,1}$ with $1$ repeating $k$ times. Thus
    $$c(\O_E(-1)\otimes\pi_E^*S)=\sum_{k=0}^n\sum_{i=0}^k(-1)^k\binom{n-k+i}{i}\bar\sigma_{1^{k-i}}\zeta^i.$$

    When $n=2$, we then have $s_1(\O_E(-1)\otimes\pi_E^*\S)=\bar\sigma_1+2\zeta$ and $s_2(\O_E(-1)\otimes\pi_E^*\S)=\bar\sigma_2+3\bar\sigma_1\zeta+3\zeta^2$. Therefore, as a class in $\A(B)$, 
    \begin{align}\label{geo24}
        \left[\tilde D_1(2,4)\cap E\right]=&j_*\left(\left\{c(\pi_E^*\Q)s(\O_E(-1)\otimes\pi_E^*\S)\right\}_1\right)\\\nonumber
        =&j_*\left(c_1(\pi_E^*\Q)+s_1(\O_E(-1)\otimes\pi_E^*\S)\right)\\\nonumber
        =&j_*(2\bar\sigma_1+2\zeta),
    \end{align}
    and
    \begin{align}\label{geo25}
        \left[\tilde D_1(2,5)\cap E\right]=&j_*\left(\left\{c(\pi_E^*\Q)s(\O_E(-1)\otimes\pi_E^*\S)\right\}_2\right)\\ \nonumber
        =&j_*\left(c_1(\pi_E^*\Q)s_1(\O_E(-1)\otimes\pi_E^*\S)+c_2(\pi_E^*\Q)+s_2(\O_E(-1)\otimes\pi_E^*\S)\right)\\ \nonumber
        =&j_*(3\bar\sigma_2+\bar\sigma_{11}+5\bar\sigma_1\zeta+3\zeta^2).
    \end{align}
    
    Since the intersection of the two varieties is of the expected dimension, and the multiplicity should be a constant integer $m$ across the intersection, we have $[E][\tilde D_1]=m[E\cap\tilde D_1]$. By comparing the coefficients in (\ref{alg24}), (\ref{alg25}), (\ref{geo24}), and (\ref{geo25}), we conclude that $m=1$ and the desired formulae follow.
\end{proof}

\subsection{The class of $\tilde\Gamma$}\hfill\\
    We now turn our attention to the class of $\Gamma=:=C\times C$ for some curve $C\subset G:=\Gr(n+1,N+1)=\GGr(n,N)$ and its proper transform $\tilde\Gamma=\tilde\Gamma:=\pi^{-1}_*\Gamma$ in $\A(B)$. In particular, we focus on the case of $C=\gamma(X)$, the Gauss image of a smooth curve.
    
    We start with the general case $\Gamma:=C\times C$, where $C\subset G$ is a smooth one-parameter family of lines in $\P^N$. We know that $[\tilde\Gamma]=\pi^*([\Gamma])+j_*(\nu)$ for some $\nu\in\A_{\dim\Gamma}(E)$. Observe that $[\Gamma]=[C]\otimes[C]$, where $C=d^\vee\sigma_{N-1,N-2}$ for some integer $d^\vee$, and $\nu\in\A_{\dim\Gamma}(E)=\A_2(E)=\A^{4N-7}(E)$. Furthermore, by Lemma \ref{projbun} we see that $\A(E)$ consists of the polynomials in $\pi_E^*(\A(\Delta_G))[\zeta]$ with degrees $\le2N-2$. Thus we have
    $$[\tilde\Gamma]=(d^\vee)^2\pi^*\left(\sigma_{N-1,N-2}\otimes\sigma_{N-1,N-2}\right)+\sum_{2N-5\le|a|\le4N-7}m_aj_*\left(\bar\sigma_a\zeta^{4N-7-|a|}\right)$$
    for some integer $m_a$'s.
    
    Similar to the way we handled $[\tilde D_1]$, to solve $m_a$'s, we consider the class $[E\cap\tilde\Gamma]$ and $[E][\tilde\Gamma]$. They must be the same class up to a multiplicity. On the one hand, via Lemma \ref{blring} and the assumption that $N\ge3$, we see that
    
    \begin{align}\label{gamalg}
        [E][\tilde\Gamma]=&\left((d^\vee)^2\pi^*\left(\sigma_{N-1,N-2}\otimes\sigma_{N-1,N-2}\right)+\sum_{2N-5\le|a|\le4N-7}m_aj_*\left(\bar\sigma_a\zeta^{4N-7-|a|}\right)\right)j_*(1).\\
        =&\sum_{2N-5\le|a|\le4N-7}-m_aj_*\left(\bar\sigma_a\zeta^{4N-6-|a|}\right)\nonumber
    \end{align}
    
    On the other hand, we can compute the class $\tilde\Gamma\cap E$ with geometric interpretation based on the following lemma.
    
    \begin{lem}
        Given any $n$-dimensional smooth family $C\subset G$ of lines in $\P^N$. Let $E$ and $\Gamma$ be defined as above, and $\tilde\Gamma$ be the proper transform of $\tilde\Gamma$. Then, under the identification $E\cong\P(\T_G)$ from Lemma \ref{nortan}, we have $\tilde\Gamma\cap E\cong\P(\T_C)$.
    \end{lem}
    
    \begin{proof}
        Let $p:\T_{ G\times G}|_{\Delta_G}\to \N_{\Delta_G, G\times G}:=\T_{ G\times G}|_{\Delta_G}/\T_{\Delta_G}$ be the quotient map. Observe that, by the construction of blowup, $\tilde\Gamma\cap E$ is the projectivization of $p(\T_\Gamma|_{\Delta_G\cap\Gamma})=p(\T_\Gamma|_{\Delta_\Gamma})$, where $\Delta_\Gamma:=\Gamma\cap\Delta_G$ is the diagonal of $\Gamma:=C\times C$. That is, $\tilde\Gamma\cap E$ is the projectivization of the vector bundle
        $$p(\T_\Gamma|_{\Delta_\Gamma})=\left\{((x,x),p(0,v)):x\in C,v\in \T_{x,C}\right\}$$
        over $\Delta_\Gamma$.
        
        Let $\iota:G\to\Delta_G\subset G\times G$ be the isomorphism from $G$ to the diagonal $\Delta_G$ of $ G\times G$, and $f:\iota^*\N_{\Delta_G, G\times G}\to\T_G$ be the isomorphism of bundles defined in Lemma \ref{nortan}. Then we have $p(\T_\Gamma|_{\Delta_\Gamma})=f^{-1}(\T_C)$, where $\T_C$ is treated as a subbundle of $\T_G|_C$. Thus $\tilde\Gamma\cap E\cong\P(\T_C)$.
    \end{proof}
    
    We proceed with the identifications $E\cong\P(\T_G)$ and $\tilde\Gamma\cap E\cong\P(\T_C)$.
    
    Consider the diagram
    \begin{align}\label{tiga}
        \xymatrixcolsep{2pc}
        \xymatrix{
        \P(\T_C)\ar[r]^-{d\gamma}\ar[rd]&\P(\gamma^*\T_G)\ar[r]^-{\tilde\gamma}\ar[d]_-{\pi_C}
        &E\cong\P(\T_G)\ar[d]^-{\pi_E}\\
        &C\ar[r]^-\gamma&G,
        }
    \end{align}
    where $\gamma:C\hookto G$ is the embedding of $C$, $\pi_C$ is the projection map and $\tilde\gamma:\P(\gamma^*\T_G)\to E$ is the map induced by the natural embedding from the pullback bundle $\gamma^*\T_G$ to $\T_G$.

    We first recall the following lemma from \cite[Proposition 9.13]{3264}.

    \begin{lem}\label{subproj}
        Let $Y$ be a smooth projective variety and $\L\subset\F$ be vector bundles over $Y$ of ranks $s$ and $r$ respectively, then
        \begin{align*}
            [\P(\L)]=\sum_{i=0}^{r-s}p^*c_i(\F/\L)\zeta_\F^{r-s-i}\in\A(\P(\F)),
        \end{align*}
        where $p:\P(\F)\to Y$ is the projection map and $\zeta_\F$ is the hyperplane section class of $\P(\F)$.
    \end{lem}

    For the case of $X$ being a curve, the following lemma will help us to compute the class $[E\cap\tilde\Gamma]$ and several other crucial classes in the following sections.

    \begin{lem}\label{projcur}
        Let $X$ be a genus $g$ smooth curve in a smooth variety $Y$, with the inclusion map $\iota:X\hookto Y$. Suppose that $\F\to Y$ is a rank $r$ vector bundle over $Y$ so that there is an inclusion $\T_X\hookto\iota^*\F$ of vector bundles over $X$. Consider the diagram
        \begin{align*}
            \xymatrixcolsep{2pc}
            \xymatrix{
            \P(\T_X)\ar@{^{(}->}[r]\ar[rd]&\P(\iota^*\F)\ar@{^{(}->}[r]^-{\tilde\iota}\ar[d]_p
            &\P(\F)\ar[d]^{\pi_\F}\\
            &X\ar@{^{(}->}[r]^-{\iota}&Y,
            }
        \end{align*}
        where $\tilde\iota:\P(\iota^*\F)\to\P(\F)$ is the canonical inclusion induced by $\iota$. Then
        $$\tilde\iota_*[\P(\T_X)]=(\pi_\F^*([\iota(X)]))\zeta_\F^{r-1}+\left(\pi_\F^*(c_1(\F)[\iota(X)]+\iota_*(K))\right)\zeta_\F^{r-2}.$$
    \end{lem}

    \begin{proof}
        Let $K$ be the canonical class of $X$ and $\zeta_\F$ be the class of hyperplane sections of $\P(\F)$. Then by Lemma \ref{subproj} we have
        \begin{align*}
            [\P(\T_X)]=&\sum_{i=0}^{r-1}\left(\tilde\iota^*\zeta_\F^{r-1-i}\right)\left(p^*c_i(\tilde\iota^*\F/\T_X)\right)\\
            =&\tilde\iota^*\zeta_\F^{r-1}+\left(\tilde\iota^*\zeta_\F^{r-2}\right)\left(p^*c_1(\tilde\iota^*\F/\T_X)\right)\\
            =&\tilde\iota^*\zeta_\F^{r-1}+\left(\tilde\iota^*\zeta_\F^{r-2}\right)p^*(\tilde\iota^*c_1(\F)+K).
        \end{align*}

        Thus by the push-pull formula we conclude
        \begin{align*}
            \tilde\iota_*[\P(\T_X)]=&\tilde\iota_*\left(\tilde\iota^*\zeta_\F^{r-1}+\left(\tilde\iota^*\zeta_\F^{r-2}\right)p^*(\tilde\iota^*c_1(\F)+K)\right)\\
            =&\zeta_\F^{r-1}\tilde\iota_*(1)+\zeta_\F^{r-2}(\pi_\F^*c_1(\F))\tilde\iota_*(1)+\zeta_\F^{r-2}(\tilde\iota_*p^*(K))\\
            =&\zeta_\F^{r-1}(\pi_\F^*([\iota(X)]))+\zeta_\F^{r-2}(\pi_\F^*(c_1(\F)[\iota(X)])+\zeta_\F^{r-2}(\pi_\F^*\iota_*(K)).
        \end{align*}
        
    \end{proof}
    
    Applying the lemma above to diagram (\ref{tiga}), we conclude that, as a class in $\A(\Bl_{\Delta_G}( G\times G))$, with $G=\Gr(2,N+1)$,
    \begin{align}\label{gamgeo}
        [E\cap\tilde\Gamma]=&j_*\left([C]\zeta^{2N-3}+\pi_E^*(c_1(\T_G)[C]+\gamma_*(K))\zeta^{2N-4}\right)\\\nonumber
        =&d^\vee j_*\left(\bar\sigma_{N-1,N-2}\zeta^{2N-3}\right)+j_*\left(\left(d^\vee\bar\sigma_{N-1,N-2}\pi_E^*(c_1(\T_G))+(2g-2)\bar\sigma_{N-1,N-1}\right)\zeta^{2N-4}\right),
    \end{align}
    where $K$ is the canonical class of the curve $C$.

    Now, we are ready to compute $[\tilde\Gamma]$ for projective curves.

    \begin{lem}\label{tilgam}
        Let $C\subset G$ be a one-parameter family of lines in $\P^N$ of genus $g$. Then
        \begin{enumerate}[(i)]
            \item for $N=3$,
            $$[\tilde\Gamma]=(d^\vee)^2\pi^*(\sigma_{21}\otimes\sigma_{21})-j_*\left(d^\vee \bar\sigma_{21}\zeta^2+(4d^\vee+2g-2)\bar\sigma_{22}\zeta\right);$$
            \item for $N=4$,
            $$[\tilde\Gamma]=(d^\vee)^2\pi^*(\sigma_{32}\otimes\sigma_{32})- j_*\left(d^\vee\bar\sigma_{32}\zeta^4+(5d^\vee+2g-2)\bar\sigma_{33}\zeta^3\right),$$
        \end{enumerate}
        where $d^\vee$ is the coefficient of the class $[C]\in\A(G)$. In particular, if $X\subset\P^N$ be a degree $d$ genus $g$ smooth curve with its Gauss map $\gamma:X\to\GGr(1,N)$ a diffeomorphism to its image.
        where $d^\vee:=2d+2g-2$.
    \end{lem}

    \begin{proof}
        For $N=3$, recall that $c_1(\T_{\GGr(1,3)})=4\sigma_1$, so (\ref{gamgeo}) becomes
        \begin{align}\label{gamgeo3}
            [E\cap\tilde\Gamma]=&d^\vee j_*\left(\bar\sigma_{21}\zeta^3\right)+j_*\left((d^\vee\bar\sigma_{21}(4\bar\sigma_1)+(2g-2)\bar\sigma_{22})\zeta^2\right)\\\nonumber
            =&d^\vee j_*\left(\bar\sigma_{21}\zeta^3\right)+(4d^\vee+2g-2)j_*\left(\bar\sigma_{22}\zeta^2\right).
        \end{align}

        For $N=4$, recall that $c_1(\T_{\GGr(1,4)})=5\sigma_1$, so (\ref{gamgeo}) yields
        \begin{align}\label{gamgeo4}
            [E\cap\tilde\Gamma]=&(2d+2g-2)j_*\left(\bar\sigma_{32}\zeta^{5}\right)+j_*\left(\left((2d+2g-2)\bar\sigma_{32}(5\bar\sigma_1)+(2g-2)\bar\sigma_{33}\right)\zeta^4\right)\\\nonumber
            =&d^\vee j_*\left(\bar\sigma_{32}\zeta^{5}\right)+(5d^\vee+2g-2)j_*\left(\bar\sigma_{33}\zeta^4\right).
        \end{align}

        The lemma follows from comparing the coefficients above and the coefficients in (\ref{gamalg}) and the Lemma below.
    \end{proof}

    The following fact justifies the notation $d^\vee$ and provides a formula for this coefficient.
    
    \begin{lem}\label{dvee}
        Let $X\subset\P^N$ be a degree $d$ genus $g$ smooth nondegenerate curve. Let $d^\vee$ be the degree of its dual hypersurface $X^\vee$. Then the class of its Gauss image $\gamma(X)$ is given by
        $$[\gamma(X)]=d^\vee\sigma_{N-1,N-2}=(2d+2g-2)\sigma_{N-1,N-2}.$$
    \end{lem}

    \begin{proof}
        Let $[\gamma(X)]=\alpha\sigma_{N-1,N-2}$ for some integer $\alpha$. 
        Let $L\subset\P^N$ be a generic $(N-2)$-dimensional linear subspace.
        
        On the one hand, $\sigma_1$ is by definition the class of lines in $\P^N$ that meet $L$, so
        \begin{align*}
            \alpha\sigma_{22}=[\gamma(X)]\sigma_1=[\{T_{x,X}:L\cap T_{x,X}\ne\emptyset\}];
        \end{align*}
        equivalently, $\alpha=|\{x\in X:L\cap T_{x,X}\ne\emptyset\}|$.

        On the other hand, recall that the dual variety of a projective variety is defined by
        $$X^\vee:=\{[H]:T_{x,X}\subset H\text{ for some }x\in X\}.$$
        Observe that $\{[H]:L\subset H\}$ is a generic line in $(\P^N)^*$, so
        \begin{align*}
            d^\vee:=\deg X^\vee=|X^\vee\cap\{[H]:L\subset H\}|=|\{[H]:\<T_{x,X},L\>\subset H\}|=\alpha.
        \end{align*}
        To concludes the proof, we recall that the degree $d^\vee$ of $X^\vee$ of a smooth projective curve is given by $d^\vee=2d+2g-2$. The formula comes from \cite[Page 61]{Gel}. However, we note that there formula there is $2d-2+g$, which is a typo, but one can deduce the correct formula from the proof on that page. 
    \end{proof}

\subsection{Self-intersection of one-parameter families of lines in $\P^3$}\hfill\\
    Now that $[\tilde D_1],[\tilde\Gamma]$, and the blowup structure are established, we are ready to study the self-intersection of one-parameter families of lines in $\P^3$.

    \begin{lem}\label{CP3}
        Let $C\subset G:=\GGr(1,3)$ be an irreducible curve. Suppose that $C$ parameterizes a one-parameter family in $\P^3$ such that any two distinct lines in the family do not intersect. Then the genus $g$ of the curve $C$ must be $0$.
    \end{lem}

    \begin{proof}
        Write $[C]=d^\vee\sigma_{21}$, $\Gamma:=C\times C$ and $\tilde\Gamma\subset B:=\Bl_{\Delta_G}G\times G$ the proper transform of $\Gamma$ as usual. Then two distinct lines in $C$ do not intersect only if, set theoretically, either $\tilde\Gamma\cap\tilde D_1=\emptyset$ or $\tilde D_1\cap\tilde\Gamma\subset E\cap\tilde\Gamma$. Since $C$ is irreducible and $E\cap\tilde\Gamma=\P(\T_C)$, we must have $[\tilde D_1][\tilde\Gamma]=m[E\cap\tilde\Gamma]$ for some integer $m\ge0$.

        By (\ref{gamgeo3}),
        \begin{align}\label{gamgeo3'}
            [E\cap\tilde\Gamma]=&d^\vee j_*\left(\bar\sigma_{21}\zeta^3\right)+(4d^\vee+2g-2)j_*\left(\bar\sigma_{22}\zeta^2\right).
        \end{align}
        By Lemma \ref{tilD1}, Lemma \ref{tilgam}, and Lemma \ref{blring},
        \begin{align*}
            [\tilde D_1][\tilde\Gamma]=&\left({d^\vee}^2\pi^*(\sigma_{21}\otimes\sigma_{21})-d^\vee j_*\left(\bar\sigma_{21}\zeta^2\right)-(4d^\vee+2g-2)j_*(\bar\sigma_{22}\zeta)\right)\\
            &\left(\pi^*(\sigma_1\otimes\sigma_0+\sigma_0\otimes\sigma_1)-j_*(2)\right)\\
            =&{d^\vee}^2\pi^*(\sigma_{22}\otimes\sigma_{21}+\sigma_{21}\otimes\sigma_{22})-2d^\vee j_*\left(\bar\sigma_{22}\zeta^2\right)\\
            &-2d^\vee j_*\left(\bar\sigma_{21}\zeta^3\right)-(8d^\vee +4g-4)j_*\left(\bar\sigma_{22}\zeta^2\right).
        \end{align*}
        Thanks to Lemma \ref{pupu}, Lemma \ref{blring} (iv), Lemma \ref{nortan}, and $c_1(\T_G)=4\sigma_1$, we have
        \begin{align*}
            \pi^*(\sigma_{22}\otimes\sigma_{21}+\sigma_{21}\otimes\sigma_{22})=&\pi^*(i_*(\sigma_{21}))\\
            =&j_*\left(\bar\sigma_{21}\left(\zeta^3+\pi_E^*(c_1(\T_G))\zeta^2+O\right)\right)\\
            =&j_*\left(\bar\sigma_{21}\left(\zeta^3+4\bar\sigma_1\zeta^2+O\right)\right)\\
            =&j_*(\bar\sigma_{21}\zeta^3+4\bar\sigma_{22}\zeta^2),
        \end{align*}
        where $O$ is the terms of $\bar\sigma_a$ with $|a|>1$. Put this into the equation above, we obtain that
        \begin{align}\label{tilinter}
            [\tilde D_1][\tilde\Gamma]=&d^\vee\left(d^\vee-2\right)j_*\left(\bar\sigma_{21}\zeta^3\right)+\left(4{d^\vee}^2-10d^\vee-4g+4\right)j_*\left(\bar\sigma_{22}\zeta^2\right).
        \end{align}
        Since by assumption $d^\vee\ne0$, by comparing the coefficients in \ref{gamgeo3'} and \ref{tilinter} we obtain $m=d^\vee-2$ and $2d^\vee g=0$, so we must have $g=0$.
    \end{proof}

\subsection{Self-intersection of scrolls in $\P^3$}\label{scroll3}\hfill\\
    Before we consider the skew curve in $\P^3$, let us examine how we can apply this tool to the classical scroll construction.

    Let $C$ be a smooth genus $g$ curve and $\iota_1,\iota_2:C\hookto\P^N$ be two embeddings of $C$ into $\P^N$ with degrees $d_1$ and $d_2$ respectively. Define the scroll $S$ with respect to $(\iota_1,\iota_2)$ by
    \begin{align*}
        S:=\bigcup_{t\in C}\<\iota_1(t),\iota_2(t)\>.
    \end{align*}
    In this case, we say that $S$ is a scroll of bidegree $(d_1,d_2)$ over $C$. It is convenient to assume a mild generic condition for $S$ a follows. We said that a scroll $S$ with respect to $\iota_1,\iota_2$ is generic if $\iota_1,\iota_2$ are generic under $\PGL(N+1)$ actions. In particular, for a generic $S$ in $\P^N$ with $N\ge3$, the embedded tangent lines $T_{\iota_1(t),\iota_1(C)}$ and $T_{\iota_2(t),\iota_2(C)}$ are skew for generic $t\in C$. A scroll $S$ with respect to $\iota_1,\iota_2$ is said to be skew if $\<\iota_1(t),\iota_2(t)\>\cap\<\iota_1(s),\iota_2(s)\>=\emptyset$ for any two distinct $t,s\in C$. Note that we do not assume that $\iota_1(C)$ and $\iota_2(C)$ lie in two disjoint linear subspaces in $\P^N$ for a scroll.

    In this subsection, we answer the question that for what $(d_1,d_2)$ and genus $g$ a generic scroll in $\P^3$ is skew . 

    Observe that a scroll $S$ gives us a map $C\to G:=\GGr(1,3)$ given by $t\mapsto\<\iota_1(t),\iota_2(t)\>$. By identifying $C$ with its image in $G$, we have by definition $[C]=(\deg S)\sigma_{21}$. Therefore, to answer the above question it suffices to compute $\deg S$ and then apply Lemma \ref{CP3}.

    \begin{lem}\label{scroll}
        Let $S\subset\P^N$ be a generic scroll with bidegree $(d_1,d_2)$, then $\deg S=d_1+d_2$.
    \end{lem}

    \begin{proof}
        Let  $\iota_1,\iota_2$ be defined as above. Recall that $\deg S=d_1+d_2$ when $\iota_1$ and $\iota_2$ are embeddings into two disjoint linear subspaces given by two complete linear systems. Indeed, if $\iota_i:C\hookto\P^{N_i}\subset\P^N$ is given by very ample line bundle $\L_i$'s with $\P^{N_1}\cap\P^{N_2}=\emptyset$, then $S$ is the image of the embedding $\P(\L_1\oplus\L_2)\hookto\P^N$ under the hyperplane section bundle, and thus of degree $c_1(\L_1)+c_1(\L_2)=d_1+d_2$.
        
        The Lemma then follows from the observation that two disjoint curves $C_1,C_2\subset\P^N$ in general position with fixed degrees can be obtained by projection from $\tilde C_1,\tilde C_2\subset\P^{\tilde N}$ with the curves $\tilde C_1,\tilde C_2$ lie in disjoint linear subspaces in $\P^{\tilde N}$.
    \end{proof}

    We now arrive at the following theorem.

    \begin{thm}
        Suppose that $S\subset\P^3$ is a skew generic scroll, then $S$ is a rational normal scroll; that is, $S$ is over $\P^1$ with bidegree $(1,1)$.
    \end{thm}

    \begin{proof}
        From Lemma \ref{CP3} and Lemma \ref{scroll} we immediately get that $S$ must be over $\P^1$.
        
        Treat the embeddings $\iota_i:\P^1\hookto\P^3$ as $4\times1$ column vector value functions. The skewness of the scroll $S$ implies that, for $t,s\in\P^1$, the matrix $M=(\iota_1(t),\iota_2(t),\iota_1(s),\iota_2(s))$ degenerate if and only if $t=s$. Since we have assumed that $S$ is generic, the only solution to $\det M=0$ is $t=s$ implies that, up to $\PGL$ actions, $\iota_i$'s are of the forms $\iota_1(t_0:t_1):=(t_0:t_1:0:0)$ and $\iota_2(t_0:t_1):=(0:0:t_0:t_1)$, thus $S$ must be a rational normal scroll.
    \end{proof}

    The following example shows that the generic condition in the above theorem is necessary.

    \begin{eg}
        Let $\iota_i:\P^1\hookto\P^3$ be defined by
        \begin{align*}
            \iota_1(t_0:t_1):=\begin{pmatrix}
                t_0^2\\2t_0t_1\\t_1^2\\0
            \end{pmatrix},\qquad\iota_2(t_0:t_1):=\begin{pmatrix}
                2t_0t_1\\t_1^2\\0\\t_0^2
            \end{pmatrix},
        \end{align*}
        then $\iota_1,\iota_2$ form a skew scroll. The skewness of the scroll can be verified by straightforward computation.
    \end{eg}

    Note that, in the above example, if we set $C_i:=\iota_i(\P^1)$, write $P_i$ for the plane that contains $C_i$ and $L:=P_1\cap P_2$, then both $C_1$ and $C_2$ tangent to $L$ at $\iota_1(1:0)$ and $\iota_2(0:1)$ respectively. The line $\<\iota_1(1:0),\iota_1(0:1)\>$ tangent to $C_1$ at $\iota_1(0:1)$ and $\<\iota_2(0:1),\iota_2(1:0)\>$ tangent to $C_2$ at $\iota_2(1:0)$. Furthermore, the tangent lines $T_{\iota_1(t),C_1}\cap T_{\iota_2(t),C_2}\neq\emptyset$ for all $t\in\P^1$. In fact, up to $\PGL$ actions, the example above is the only skew scroll of bidegree $(2,2)$. It is unknown to the author if there are skew scrolls of higher bidegrees.

\subsection{The skew curve in $\P^3$}\hfill\\
    Lemma \ref{CP3} also helps us to find the algebraically skew curve in $\P^3$.

    \begin{thm}\label{P3}
        The twisted cubic is the only algebraically skew curve in $\P^3$.
    \end{thm}

    \begin{proof}
        Let $X\subset\P^3$ be a degree $d$ genus $g$ smooth spatial curve. We may assume that the curve $X$ is connected and irreducible. Indeed, if a skew spacial curve $X$ has two distinct components $X_0$ and $X_1$, then their tangent varieties $T_{X_0}$ and $T_{X_1}$ must intersect as they are both surfaces, which contradicts to the definition of skewness.

        By applying Lemma \ref{CP3} to the Gauss image of $X$, we see that $X$ being skew implies that $g=0$. Since $\P^1$ can be parameterized by homogeneous coordinates, it is then an elementary exercise to show that a genus zero spacial skew curve must be the twisted cubic. 
    \end{proof}

    \begin{rmk}Theorem \ref{P3} can be proved in several different manners, and one may remove assumption of the Gauss map being diffeomorphic in this case. The readers might consider some of these proofs simpler than the one presented above. For instance, one can project the curve to $\P^1$, then applies the Hurwitz formula to this projection. The same result also follows from a more general statement about Terracini loci of curves in $\P^3$ \cite{TerrCurve}. We keep the proof above instead of other simpler ones because the presented approach illuminates what to do in higher dimensional cases.
    \end{rmk}

\section{The skew curves in $\P^4$}\label{four}

    Now we move on to one-parameter family of lines and skew curves in $\P^4$.

\subsection{Self-intersection of one-parameter family of lines in $\P^4$}\hfill\\
    We first consider the general case of a smooth curve $C\subset G:=\GGr(1,4)$. We would like to know under what conditions two distinct lines in the family do not intersect, and if they do intersect, how many pairs of lines intersect.
    
    Set $\Gamma:=C\times C$, then $[\Gamma]\in\A^{10}( G\times G)$. Thanks to Proposition \ref{bipro}, we have $[D_1]\in\A^2( G\times G)$ and the expected dimension of the intersection $\Gamma\cap D_1$ is zero. However, we know that the diagonal $\Delta_\Gamma$ of $\Gamma$ is a curve lying in the intersection $\Gamma\cap D_1$, so we need to resolve the excess intersection on the diagonal.

    Following the approach from previous sections, we consider the proper transforms $\tilde D_1$ and $\tilde\Gamma$ in $B:=\Bl_{\Delta_G}G\times G$. This approach works for a generic one-parameter family of lines in the following sense.

    \begin{lem}\label{gCP4}
        Consider a family of lines in $\P^4$ parameterized by a smooth curve $C\subset G:=\GGr(1,4)$. Suppose that for a generic point $x\in C$ with a local affine parameterization
        \begin{align*}
            x=\begin{pmatrix}
                1&0\\0&1\\f&g
            \end{pmatrix}
        \end{align*}
        of $C$ around $x$, the column vectors
        \begin{align*}
            \begin{pmatrix}
                1\\0\\f
            \end{pmatrix},
            \begin{pmatrix}
                0\\1\\g
            \end{pmatrix},
            \begin{pmatrix}
                0\\0\\f'
            \end{pmatrix},
            \begin{pmatrix}
                0\\0\\g'
            \end{pmatrix}
        \end{align*}
        are linearly independent, where $f$ and $g$ are $3\times1$ column vectors. Then $E\cap\tilde D_1\cap\tilde\Gamma$ is not a curve, where $E$ is the exceptional subvariety of the blowup $B:=\Bl_{\Delta_G}G\times G$.
    \end{lem}

\subsection{Self-intersection of scrolls in $\P^4$}\label{scroll4}\hfill\\
    An immediate application of Lemma \ref{gCP4} is the self-intersection of generic scrolls in $\P^4$. 

    \begin{thm}
        Let $S$ be a generic bidegree $(d_1,d_2)$ scroll over a genus $g$ curve $C$ defined as in previous sections. Set $d^\vee:=d_1+d_2$. Then $S$ is skew if $(d^\vee)^2-(5d^\vee+6g-6)=0$. In particular, if a skew generic scroll $S$ with respect to $\iota_1,\iota_2:C\to\P^4$ satisfies $T_{\iota_1(t),\iota_1(C)}\cap T_{\iota_2(t),\iota_2(C)}=\emptyset$ for all $t\in C$, then $S$ is a rational normal scroll.
    \end{thm}

    \begin{proof}
        Let $\iota_1,\iota_2:C\hookto\P^4$ be the embeddings that define the scroll and $\gamma:C\to G:=\GGr(1,4)$ be given by $x\mapsto\<\iota_1(x),\iota_2(x)\>$ as before. Then, Lemma \ref{scroll} gives us $[\gamma(C)]=d^\vee\sigma_{32}$, and so, by Lemma \ref{tilgam}, we have 
        $$[\tilde\Gamma]=(d^\vee)^2\pi^*(\sigma_{32}\otimes\sigma_{32})- j_*\left(d^\vee\bar\sigma_{32}\zeta^4+(5d^\vee+2g-2)\bar\sigma_{33}\zeta^3\right).$$
        Meanwhile, from Lemma \ref{tilD1} we also have
        $$[\tilde{D_1}(2,5)]=\pi^*(\sigma_2\otimes\sigma_0+\sigma_1\otimes\sigma_1+\sigma_0\otimes\sigma_2)-j_*(5\bar\sigma_1+3\zeta).$$
        Due to the genericity assumption of $S$, we can apply Lemma \ref{gCP4} to see that the intersection $\tilde D_1\cap\tilde\Gamma$ consists of pairs of intersecting lines in $S$ as well as the points $t\in C$ such that $T_{\iota_1(t),\iota_1(C)}\cap T_{\iota_2(t),\iota_2(C)}\ne\emptyset$. Thus the theorem follows from
        \begin{align*}
            \ll\tilde D_1\rr\ll\tilde\Gamma\rr=&\ll(d^\vee)^2\pi^*(\sigma_{33}\otimes\sigma_{33})-j_*\left(5d^\vee\bar\sigma_{33}\zeta^5+3d^\vee\bar\sigma_{32}\zeta^6+3(5d^\vee+2g-2)\bar\sigma_{33}\zeta^5\right)\rr\\
            =&\left((d^\vee)^2-(5d^\vee+6g-6)\right)P,
        \end{align*}
        where $P$ is the class of a point. Indeed, by solving $(d^\vee)^2-(5d^\vee+6g-6)=0$ for integers $g\ge0$ and $d^\vee:=d_1+d_2\ge2$, we either have $g=0$ with $d^\vee=2,3$ or $g=1$ with $d^\vee=5$. Since there is no elliptic curves of degree $\le2$ in $\P^4$, the case of $g=1$ is impossible. Thus we conclude that the only skew generic scrolls in $\P^4$ are the rational normal scrolls.
    \end{proof}

    We would like this approach to work for the case of skew curves in $\P^4$. However, it turns out that the family of tangent lines of a curve in $\P^4$ does not meet the conditions of Lemma \ref{gCP4}. Thus we need a more careful study of the excess intersection on the diagonal.

\subsection{The excess intersection}\label{ExcInt}\hfill\\
    Given a curve $X\subset\P^4$. Let $\gamma:X\to G:=\GGr(1,4)$ be its Gauss map and $\Gamma:=\gamma(X)\times\gamma(X)\subset G\times G$. We would like to determine whether $X$ is skew by taking the intersection of the proper transforms $\tilde D_1$ and $\tilde\Gamma$ in the blowup space $B:=\Bl_{\Delta_G}(G\times G)$. The problem is that the intersection $\tilde\Gamma\cap\tilde D_1$ might not resolve the excess intersection on the diagonal. In fact, we will see in the following lemma that the intersection $\tilde\Gamma\cap\tilde D_1$ contains a double curve supported by $\tilde\Gamma\cap E$.

    \begin{lem}\label{theC}
        Let $X\subset\P^4$ be a smooth curve with third osculating planes nowhere degenerated. Let $G:=\Gr(2,5)$, $\Gamma:=\Gamma(X)$ and $D_1:=D_1(2,5)$ be defined as in the previous sections, $E$ be the exceptional locus of the blowup $B:=\Bl_{\Delta_G}(G\times G)$, and $\tilde\Gamma$ and $\tilde D_1$ be the proper transforms of $\Gamma$ and $D_1$ in $B$ respectively. Then, the curve $\tilde\Delta_\Gamma:=\tilde\Gamma\cap E$ is the diagonal of $\tilde\Gamma\cong\Gamma:=\gamma(X)\times\gamma(X)$ and supports a double curve component of the intersection $\tilde\Gamma\cap\tilde D_1$. More precisely, if we consider the blowup
        \begin{align}\label{BlowUp}
            \xymatrixcolsep{2pc}
            \xymatrix{
            E^\dagger\ar@{^{(}->}[r]^-{j^\dagger}\ar[d]_{\pi^\dagger_E}
            &B^\dagger:=\Bl_{\tilde\Delta_\Gamma}B\ar[d]^{\pi^\dagger}\\
            \tilde\Delta_\Gamma\ar@{^{(}->}[r]^{i^\dagger}&B,
            }
        \end{align}    
        then the intersection of the proper transforms $D_1^\dagger$ and $\Gamma^\dagger$ of $\tilde D_1$ and $\tilde\Gamma$ in $B^\dagger$ has a reduced irreducible curve component at $\Delta^\dagger_\Gamma:=E^\dagger\cap\Gamma^\dagger$, where $E^\dagger$ is the exceptional divisor of the blowup $B^\dagger$.
    \end{lem}

    \begin{proof}
        Since $\Delta_\Gamma:=\Gamma\cap\Delta_G$ is the diagonal of $\Gamma:=\gamma(X)\times\gamma(X)$, we clearly have $\Gamma^\dagger\cong\tilde\Gamma\cong\Gamma$ and $\Delta_\gamma^\dagger\cong\tilde\Delta_\Gamma\cong\Delta_\Gamma$ are their diagonals and so irreducible curves.

        To prove that $\tilde\Delta_\Gamma\subset\tilde D_1$ and $\Delta_\Gamma^\dagger\subset D_1^\dagger$, it suffices, thanks to Lemma \ref{nortan}, to show that for any given point $x_0\in X$, there exists an open neighborhood $U\subset\C$ of $0$, and curves $\alpha,\beta:U\to\Gamma$ with $\alpha(0)=\beta(0)=\gamma(x_0)$, $\{\gamma(x_0)\}\times\alpha(U)\subset\{\gamma(x_0)\}\times\gamma(X)$, and $\{\gamma(x_0)\}\times\times\beta(U)\subset(\{\gamma(x_0)\}\times G)\cap D_1$ such that the first derivative and the second derivative of $\alpha$ and $\beta$ at $0$ agree. Indeed, $\frac{\d}{\d t}\alpha(0)=\frac{\d}{\d t}\beta(0)$ implies $\tilde\alpha(0)=\tilde\beta(0)$, where  $\tilde\alpha$ and $\tilde\beta$ are the proper transforms of $\alpha$ and $\beta$ in $B$, and $\frac{\d^2}{\d t^2}\alpha(0)=\frac{\d^2}{\d t^2}\beta(0)$ implies $\frac{\d}{\d t}\tilde\alpha(0)=\frac{\d}{\d t}\tilde\beta(0)$ and so $\alpha^\dagger(0)=\beta^\dagger(0)$, where $\alpha^\dagger$ and $\beta^\dagger$ are the proper transforms of $\tilde\alpha$ and $\tilde\beta$ in $B^\dagger$.
        
        Similarly, to prove that $\Gamma^\dagger$ and $D_1^\dagger$ do not tangent along $\Delta_\Gamma^\dagger$, we show that for any curve $\rho:U\to(\{\gamma(x_0)\}\times G)\cap D_1$ with $\rho(0)=(\gamma(x_0),\gamma(x_0))$, its third derivative does not agree with the third derivative of $\alpha$.

        To construct $\alpha$ and $\beta$, we first construct local coordinates around an arbitrary point on $\Delta_\Gamma$. To begin, fix any point $x_0\in X\subset\P^4$. By choosing the coordinates of $\P^4$, $X$ has a local analytic parametrization $U\to\P^4$ from some open neighborhood $U$ of $0\in\C$ given by $x_t:=(1:t:f_2(t):f_3(t):f_4(t))$ such that $f(0)=(1:0:f_2(0):f_3(0):f_4(0))=x_0$ and $f_i(0)=f_i'(0)=0$ for $i=2,3,4$. Then, in Stiefel coordinates, we have
        \begin{align*}
            \gamma(x_t)=(x_t,x_t')=\begin{pmatrix}
                1&0\\t&1\\f(t)&f'(t)
            \end{pmatrix}=\begin{pmatrix}
                1&0\\0&1\\f(t)-tf'(t)&f'(t)
            \end{pmatrix},
        \end{align*}
        where $f(t):=\begin{pmatrix}
            f_2(t)\\f_3(t)\\f_4(t)
        \end{pmatrix}$. In other words, in the affine coordinates of $\Gr(2,5)$ with the top minor matrix being the identity, we have $\gamma(f(t))=(f(t)-tf'(t),f'(t))$ and $\gamma(f(0))=O$ is the zero matrix.

        Set $\alpha(t):=\gamma(f(t))$ and $\beta(t):=(-\frac{t}{2}f'(t),f'(t))$, then straightforward computation shows that $\alpha^{(k)}(0)=\beta^{(k)}(0)$ for $k=0,1,2$. Clearly we have $\{\gamma(x_0)\}\times\alpha(U)\subset\{\gamma(x_0)\}\times\gamma(X)$, and $(O,\beta(t))\in D_1$ for any $t\in U$ because $\rank\beta(t)\le1$. This shows that $\Gamma^\dagger$ and $D_1^\dagger$ intersect along $\Delta_\Gamma^\dagger$.

        Now, suppose for contradiction that there exists $\rho:U\to(\{\gamma(x_0)\}\times G)\cap D_1$ with $\rho(0)=(\gamma(x_0),\gamma(x_0))$ such that $\alpha^{(k)}(0)=\rho^{(k)}(0)$ for $k=0,1,2,3$. Since $(\gamma(x_0),\rho(t))\in D_1$, $\rank\rho(t)\le1$ for any $t\in U$. Thus we may write
        \begin{align*}
            \rho(t):=(a(t)g(t),g(t))
        \end{align*}
        for some analytic function $a(t)$ and some $3\times1$ column vector $g(t)$ of analytic functions. Then, since $O=\alpha(0)=\rho(0)$, we have $g(0)=0$. Thus
        \begin{align*}
            (0,f''(0))=\alpha'(0)=\rho'(0)=(a(0)g'(0),g'(0)).
        \end{align*}
        Since we have assumed that the third osculating plane of $X$ is nowhere degenerated, $g'(0)=f''(0)\ne0$, so we must have $a(0)=0$. Hence
        \begin{align*}
            (-f''(0),f'''(0))=\alpha''(0)=\rho''(0)=(2a'(0)g'(0),g''(0)).
        \end{align*}
        In particular, plug $g'(0)=f''(0)$ and $g''(0)=f'''(0)$ into $\alpha'''(0)=\rho'''(0)$, we obtain
        \begin{align*}
            (-2f'''(0),f''''(0))=\alpha'''(0)=\rho'''(0)=(3h''(0)f''(0)+3h'(0)f'''(0),g'''(0)).
        \end{align*}
        This implies that $f'''(0)$ is a multiple of $f''(0)$, which contradicts to the assumption that the third osculating plane of $X$ at $x_0$ being nondegenerate as the first three derivatives of $f$ at $x_0$ are linearly dependent. This concludes the proof.
    \end{proof}

    Although the above lemma shows that the blowups we did are not enough to resolve the excess intersection component in $\Gamma\cap D_1$ along the diagonal, the blowups do smooth the variety $D_1$, see Section \ref{decomp}. In addition, since the excess intersection component between the proper transforms $D_1^\dagger$ and $\Gamma^\dagger$ is now a smooth reduced irreducible curve, we may attempt to apply the excess intersection formula.

    \begin{thm}[The symmetric excess intersection formula, cf. \cite{3264}]
        Let $B$ be a smooth projective variety, and $R,D\subset Y$ be two locally complete intersection subvarieties of $Y$, then, as a class in $\A(Y)$, we have
        $$[R][D]=\sum_{C}(\iota_C)_*(\epsilon_C),$$
        where the sum is taken over the connected components $C$ of $R\cap D$, $\iota_C:C\to Y$ denotes the inclusion, and
        $$\epsilon_C:=\{s(C,Y)c(\N_{R,Y}|_C)c(\N_{D,Y}|_C)\}_{\dim[R][D]}$$
        with $s(C,Y)$ being the Segre class of $C$ in $Y$, which, when $C$ is again a locally complete intersection in $Y$, is given by $c(\N_{C,Y})^{-1}$.
    \end{thm}

    From the excess intersection formula and the lemma above, we see that the skewness of a curve $X\subset\P^4$ implies that $[D_1^\dagger][\Gamma^\dagger]-(\iota_{\Delta^\dagger_\Gamma})_*\epsilon_{\Delta^\dagger_\Gamma}=0$. In particular, we would have
    \begin{align*}
        \ll D_1^\dagger\rr\ll\Gamma^\dagger\rr-\ll(\iota_{\Delta^\dagger_\Gamma})_*\epsilon_{\Delta^\dagger_\Gamma}\rr=0
    \end{align*}
    in the numerical group $\NN(B^\dagger)$.
    
    Since the expected dimension of $\Gamma^\dagger\cap D^\dagger_1$ is $0$, both $\Gamma^\dagger$ and $D^\dagger_1$ are smooth, and the component $\Delta^\dagger_\Gamma$ is an irreducible smooth curve, we have
    \begin{align*}
        \epsilon_{\Delta^\dagger_\Gamma}=&\left\{{s(\Delta^\dagger_\Gamma,B^\dagger)c(\N_{\Gamma^\dagger,B^\dagger}|_{\Delta^\dagger_\Gamma})c(\N_{D^\dagger_1,B^\dagger}|_{\Delta^\dagger_\Gamma})}\right\}_0\\        =&c_1(\N_{\Gamma^\dagger,B^\dagger}|_{\Delta^\dagger_\Gamma})+c_1(\N_{D^\dagger_1,B^\dagger}|_{\Delta^\dagger_\Gamma})-c_1(\N_{\Delta^\dagger_\Gamma,B^\dagger})\\
        =&\iota_{\Delta^\dagger_\Gamma}^*c_1(\T_{B^\dagger})-\iota_{\Delta^\dagger_\Gamma\Gamma^\dagger}^*c_1(\T_{\Gamma^\dagger})+\iota_{\Delta^\dagger_\Gamma}^*c_1(\T_{B^\dagger})-\iota_{\Delta^\dagger_\Gamma D^\dagger_1}^*c_1(\T_{D^\dagger_1})-\left(\iota_{\Delta^\dagger_\Gamma}^*c_1(\T_{B^\dagger})-c_1(\T_{\Delta^\dagger_\Gamma})\right)\\
        =&\iota_{\Delta^\dagger_\Gamma}^*c_1(\T_{B^\dagger})-\iota_{\Delta^\dagger_\Gamma\Gamma^\dagger}^*c_1(\T_{\Gamma^\dagger})-\iota_{\Delta^\dagger_\Gamma D^\dagger_1}^*c_1(\T_{D^\dagger_1})+c_1(\T_{\Delta^\dagger_\Gamma}),
    \end{align*}
    where $\iota$'s are the appropriate inclusions. Our task then becomes to compute the classes $\ll D_1^\dagger\rr,\ll\Gamma^\dagger\rr$, the first Chern classes above, and the numerical classes of their pullbacks in $\NN(\Delta^\dagger_\Gamma)$.

    \subsection{The factorization of $\tilde D_1\to D_1$}\label{decomp}\hfill\\
    In this section and the next section, we compute the class $[D_1^\dagger]$ in $\A(B^\dagger)$. To do this, we need to compute the class $[E^\dagger\cap D_1^\dagger]\in\A(E^\dagger)$. The class $[E^\dagger\cap D_1^\dagger]$ can be interpreted as the class induced by the projective subbundle $\P(\N_{\tilde\Delta_\Gamma,\tilde D_1})\subset\P(\N_{\tilde\Delta_\Gamma,B})\cong E^\dagger$ as long as $\tilde D_1$ is smooth. Thus, to compute $[D_1^\dagger]$, we would need to verify that $\tilde D_1$ is smooth, compute the canonical class $c_1(\T_{\tilde D_1})$ and its pullback in $\A(\tilde\Delta_\Gamma)$.
    
    For the structure of $\tilde D_1$, we first consider the incidence correspondence
    \begin{align*}
        \hat D_1:=\left\{(L,U,W):L\subset U\cap W\right\}\subset\P^4\times\Gr(2,V)\times\Gr(2,V),
    \end{align*}
    where $V$ is the $5$-dimensional vector space with $X\subset\P^4=\P(V)\cong\Gr(1,V)$.
    
    Write the restriction of the blowdown map $\pi:B\to G\times G$ to $\tilde D_1$ by $\pi_D:=\pi|_{\tilde D_1}:\tilde D_1\to D_1$. Then the birational morphism $\pi_D$ can be factorized into
        $\xymatrixcolsep{2pc}
        \xymatrix{
        \pi_D:\tilde D_1\ar[r]^-{\tilde\pi_D}
        &\hat D_1\ar[r]^-{\hat\pi_D}&D_1,
        }$
     where $\hat\pi_D$ is the projection from $\hat D_1$ to $\Gr(2,V)\times\Gr(2,V)$ and $\tilde\pi_D$ is the unique choice of morphism for the factorization. Explicitly, $\tilde\pi_D$ is the identity map outside the exceptional locus $E$, and under the identification from Lemma \ref{nortan}, the map
     \begin{align*}
         \tilde\pi_D|_{E\cap\tilde D_1}:E\cap\tilde D_1\cong\P(\T_G)\cap\tilde D_1\cong\P(\Hom(\S,\Q))\cap\tilde D_1\longto\hat D_1
     \end{align*}
     is given by
     \begin{align*}
         \tilde\pi_D|_{E\cap\tilde D_1}:(x,[v])\longmapsto(\ker v,x,x),
     \end{align*}
     where $(x,[v])\in\P(\Hom(\S,\Q))\cap\tilde D_1$ is treated as an element in the projective bundle $\P(\Hom(\S,\Q))$, with $v$ a nonzero map from the universal subbundle $\S$ to the universal quotient bundle $\Q$. The map $v$ indeed has an one dimensional kernel as $(x,[v])\in\P(\Hom(\S,\Q))\cap\tilde D_1$, see formula (\ref{nortand1}).

     To understand the structure of $\hat D_1$, we observe that the incidence variety
     $$\{(L,U):L\subset U\}\subset\P^4\times\Gr(2,V)$$
     is by definition isomorphic to the flag variety $\Fl(1,2,V)$ and the projective bundle $\P(\Q)$, where $\Q$ denotes the universal quotient bundle over $\P^4$. Thus we conclude the following properties of $\hat D_1$.
     
     \begin{prop}
         Let $\xymatrixcolsep{2pc}
        \xymatrix{
        \pi_D:\tilde D_1\ar[r]^-{\tilde\pi_D}
        &\hat D_1\ar[r]^-{\hat\pi_D}&D_1
        }$ be the factorization of $\pi_D$ described as above and $\Q$ be the universal quotient bundle of $\P^4$.
        \begin{enumerate}[(i)]
            \item The variety $\hat D_1$ is isomorphic to the fibre product of two copies of the projective bundle $\P(\Q)$ over $\P^4$, that is, $\hat D_1\cong\P(\Q)\times_{\P^4}\P(\Q)$; furthermore, under this identification, the bundle map $p:\hat D_1\to\P^4$ is given by the projection from $\hat D_1\subset\P^4\times\Gr(2,V)\times\Gr(2,V)$ to its first component.
            \item $\hat D_1$ is smooth.
            \item The variety $\hat\Delta_D:=\hat\pi_D^{-1}(\Delta_G)$ is isomorphic to the flag variety $\Fl(1,2,V)$ as well as the projective bundle $\P(\T_{\P^4})\cong\P(\Q)$, and the inclusion $i_D:\hat\Delta_D\cong\P(\Q)\hookto\hat D_1\cong\P(\Q)\times_{\P^4}\P(\Q)$ is given by the fibrewise inclusion of diagonals.
            \item Denote the rank one and two universal subbundles of the flag variety $\hat\Delta_D\cong\Fl(1,2,V)$ by $\S_1$ and $\S_2$ respectively and the trivial bundle induced by the ambient space by $\V$, then the bundle $\S_1$ is the universal line bundle of $\P^4$ pulled back to $\hat\Delta_D$ by $\hat\pi_\Delta:\hat\Delta_D\to\P^4$, and $\S_2/\S_1$ is the universal line bundle of the projective bundle $\hat\Delta_D\cong\P(\Q)$.
        \end{enumerate}
     \end{prop}
     
    The lemma below is convenient for the desired Chow ring of the bi-projective bundle $\hat D_1$.

     \begin{lem}\label{propro}
         Suppose that $Y$ is a smooth variety, and $\F,\L$ are two vector bundles over $Y$. Write $p_\F:\P(\F)\to Y$ and $p_\L:\P(\L)\to Y$, then there are canonical isomorphisms between the induced bi-projective bundle and the top spaces of the two corresponding projective towers; that is,
         $$\P(\F)\times_Y\P(\L)\cong\P(p_\F^*\L)\cong\P(p_\L^*\F).$$
         Furthermore, 
         $$\A(\P(\F)\times_Y\P(\L))=\frac{\A(Y)\left[\zeta_\F,\zeta_\L\right]}{\left(\sum_{i=0}^{\rank\F}\zeta_\F^{\rank\F-i}c_i(\F),\sum_{i=0}^{\rank\L}\zeta_\L^{\rank\L-i}c_i(\L)\right)}$$
         where, under the canonical identification $\P(\F)\times_Y\P(\L)\cong\P(p_\F^*\L)$, $\zeta_\F$ is the the class of hyperplane sections of $\P(\F)$ pullback to $\P(p_\F^*\L)$, and $\zeta_\L$ is the the class of hyperplane sections of $\P(p_\F^*\L)$.
     \end{lem}

     \begin{proof}
         The first assertion follows immediately from \cite[Lemma 3.8]{protower}, and the ``furthermore" part follows from \cite[Theorem 9.6]{3264}.
     \end{proof}

    \begin{prop}\label{blowupp}
        Let $i_D:\hat\Delta_D\cong\P(\Q)\hookto\hat D_1\cong\P(\Q)\times_{\P^4}\P(\Q)$ be the fibrewise diagonal inclusion, then the corresponding Chow rings satisfy the following properties.
        \begin{enumerate}[(i)]
            \item The Chow ring of $\hat\Delta_D\cong\P(\Q)$ is given by
            \begin{align*}
               \A(\hat\Delta_D)=\frac{\Z\left[e,\zeta_\Q\right]}{\left(e^5,\sum_{i=0}^4\zeta_\Q^{4-i}e^i\right)},
            \end{align*}
            where $e$ is the class of hyperplanes of $\P^4$ pulled back by $\hat\pi_\Delta:=\hat\pi_D|_{\hat\Delta_D}:\hat\Delta_D\to\P^4$ and $\zeta_\Q$ is the class of hyperplane sections of the projective bundle $\hat\Delta_D\cong\P(\Q)$.
            \item The Chow ring of $\hat D_1\cong\P(\Q)\times_{\P^4}\P(\Q)$ is given by
            \begin{align*}
               \A(\hat D_1)=\frac{\Z\left[e,\zeta_1,\zeta_2\right]}{\left(e^5,\sum_{i=0}^4\zeta_1^{4-i}e^i,\sum_{i=0}^4\zeta_2^{4-i}e^i\right)},
            \end{align*}
            where $e$ is the class of hyperplanes of $\P^4$ pulled back by $\hat\pi_D:\hat D_1\to\P^4$ and $\zeta_1,\zeta_2$ are the classes of hyperplane sections of the two components of the fibre product of projective bundles $\hat D_1\cong\P(\Q)\times_{\P^4}\P(\Q)$.
            \item The pullback $i_D^*:\A(\hat D_1)\to\A(\hat\Delta_D)$ is given by $e\mapsto e$ and $\zeta_i\mapsto\zeta_\Q$ for $i=1,2$.
        \end{enumerate}
    \end{prop}

    Now we study the map $\tilde D_1\to\hat D_1$. Observe that the normal bundle $\N_{\hat\Delta_D,\hat D_1}$ is isomorphic to $\Hom(\S_2/\S_1,\V/\S_2)$, then observe that $\P(\Hom(\S_2/\S_1,\V/\S_2))$ is naturally isomorphic to $\tilde D_1\cap E\subset\P(\T_G)$, given by
    \begin{align*}
        (L,U;[v])\longmapsto(U,\tilde v),
    \end{align*}
    where $(L,U)\in\hat\Delta_D\cong\Fl(1,2,V)$ and $[v]\in\P(\Hom(\S_2/\S_1,\V/\S_2)_{(L,U)})$ is treated as the class of a map $v\in\Hom(\S_2/\S_1,\V/\S_2)_{(L,U)}$ and $\tilde v\in\P(\T_G)\cong\P(\Hom(\S,\V/\S)_U)$ is the unique class of maps from the fibre $\S_U$ of the universal subspace bundle of $G:=\Gr(2,V)$ at $U$ to the fibre $(\V/\S)_U$ of the universal quotient bundle at $U$ that induces the map $v$.
    
    Thus we have the blowup diagram
     \begin{align}\label{blowupd}
        \xymatrixcolsep{2pc}
        \xymatrix{
        \tilde\Delta_D\ar@{^{(}->}[r]^-{j_D}\ar[d]_{\tilde\pi_\Delta}
        &\tilde D_1\cong\Bl_{\hat\Delta_D}\hat D_1\ar[d]^{\tilde\pi_D}\\
        \hat\Delta_D\ar@{^{(}->}[r]^-{i_D}&\hat D_1\cong\P(\Q)\times_{\P^4}\P(\Q).
        }
    \end{align}

    To compute $c_1(\T_{\tilde D_1})$, we recall the formula of the first Chern classes of projective bundles and blowup spaces, cf. \cite[page 608]{grif_harr}.
    
    \begin{lem}\label{cTP}
        Let $Y$ be an $k$-dimensional smooth variety and $\F\to Y$ be a rank $r$ vector bundle. Consider the projective bundle $\pi_\F:\P(\F)\to Y$, then
        $$c_1(\T_{\P(\F)})=r\zeta_\F+\pi_\F^*(c_1(\F)+c_1(\T_Y)),$$
        where $\zeta_\F$ is the class of hyperplane sections of $\P(F)$.
    \end{lem}

    \begin{lem}\label{blowupch}
        Suppose that $Y$ is a smooth variety with a smooth codimension $r$ subvariety $i_Z:Z\hookto Y$. Denote $\pi_Y:\tilde Y:=\Bl_ZY\to Y$ the blowup of $Y$ along $Z$ with the exceptional locus $\tilde Z$, $\pi_Z:=\pi_Y|_{\tilde Z}$, and the inclusion $j_Z:\tilde Z\hookto\tilde Y$, so we have the blowup diagram
        \begin{align*}
            \xymatrixcolsep{2pc}
            \xymatrix{
            \tilde Z\ar@{^{(}->}[r]^-{j_Z}\ar[d]_{\pi_Z}
            &\tilde Y\ar[d]^{\pi_Y}\\
            Z\ar@{^{(}->}[r]^{i_Z}&Y.
            }
        \end{align*}
        Then
        \begin{align*}
            c_1(\T_{\tilde Y})=\pi_Y^*c_1(\T_Y)+{j_Z}_*(1-r).
        \end{align*}
    \end{lem}    

    Applying Lemma \ref{propro} and Lemma \ref{cTP} twice, we get $c_1(\T_{\tilde\Delta_D})=4\zeta_\Q+6e$ and $c_1(\T_{\hat D_1})=4\zeta_1+4\zeta_2+7e$. Plug this back into Lemma \ref{blowupch}, we then conclude that
    \begin{align}\label{TtilD}
        c_1(\T_{\tilde D_1})=\pi_D^*(4\zeta_1+4\zeta_2+7e)-{j_D}_*(2).
    \end{align}

    \subsection{The class $\ll D_1^\dagger\rr$}\hfill\\
     Now, we shift our attention to $[\tilde\Delta_\Gamma]\in\A(\tilde D_1)$. Since $\tilde\Delta_\Gamma:=E\cap\Gamma\subset E\cap\tilde D_1=\tilde\Delta_D$ is in the exceptional locus of the blowup space $\tilde D_1=\Bl_{\hat\Delta_D}\hat D_1$, it is suffice to compute $[\tilde\Delta_\Gamma]$ as a class in $\A(\tilde\Delta_D)\cong\A(\P(\Hom(\S_2/\S_1,\V/\S_2)))$.

    Consider the map $\hat\gamma:X\to\hat\Delta_D\cong\Fl(1,2,V)$ given by $x\mapsto(x,T_{x,X})$, where $x$ and $T_{x,X}$ are treated as one and two dimensional subspaces in $V$. Recall that $\T_{\hat\Delta_D}=\Hom(\S_1,\S_2/\S_1)\oplus\Hom(\S_1,\V/\S_2)\oplus\Hom(\S_2/\S_1,\V/\S_2)$, cf. \cite{flag} and consider the diagram
    \begin{align*}
        \xymatrixcolsep{2pc}
        \xymatrix{
        \T_X\ar[r]\ar[dr]&\hat\gamma^*\Hom(\S_2/\S_1,\V/\S_2)\ar[d]\ar[r]&\Hom(\S_2/\S_1,\V/\S_2)\ar[d]\\
        &X\ar[r]^-{\hat\gamma}&\hat\Delta_D,
        }
    \end{align*}
    where the top left map is given by the composition of the pushforward map between tangent bundles and the projection. Then the top row of the diagram induces a map $\tilde\gamma:\P(\T_X)\to\P(\Hom(\S_2/\S_1,\V/\S_2))\cong\tilde\Delta_D$, which is an inclusion with its image being the curve $\tilde\Delta_\Gamma:=\tilde\Gamma\cap E$.

    Together with Lemma \ref{projcur}, we deduce that
    \begin{align}\label{tildeX}
        [\tilde\Delta_\Gamma]=&(\tilde\pi_\Delta^*([\hat\gamma(X)]))\zeta_\H^2+\left(\tilde\pi_\Delta^*(c_1(\Hom(\S_2/\S_1,\V/\S_2))[\hat\gamma(X)])+\hat\gamma_*(K)\right)\zeta_\H\\\nonumber
        =&(\tilde\pi_\Delta^*([\hat\gamma(X)]))\zeta_\H^2+\left(\tilde\pi_\Delta^*(c_1(\Hom(\S_2/\S_1,\V/\S_2))[\hat\gamma(X)])+(2g-2)e^4\zeta_\Q^3\right)\zeta_\H
    \end{align}
    where $K$ denotes the canonical class of $X$, $\zeta_\H$ the class of hyperplanes of the projective bundle $\tilde\Delta_D\cong\P(\Hom(\S_2/\S_1,\V/\S_2))$, the map $\tilde\pi_\Delta$ is defined as in diagram (\ref{blowupd}), and $e^4\zeta_\Q^3$ the class of a point in $\A(\hat\Delta_D)$.

    By Proposition \ref{blowupp}, and the sequences
    \begin{align*}
        \xymatrixcolsep{2pc}
        \xymatrix{
        0\ar[r]&\S_1\ar[r]&\S_2\ar[r]&\S_2/\S_1\ar[r]&0
        }
    \end{align*}
    and
    \begin{align*}
        \xymatrixcolsep{2pc}
        \xymatrix{
        0\ar[r]&\S_2\ar[r]&\V\ar[r]&\V/\S_2\ar[r]&0,
        }
    \end{align*}
    we have $c_1(\S_1)=-e$, $c_1(\S_2/\S_1)=-\zeta_\Q$, $c_1(\S_2)=-e-\zeta_\Q$, and $c_1(\V/\S_2)=e+\zeta_\Q$. Thus
    \begin{align}\label{Hch}
        c_1(\Hom(\S_2/\S_1,\V/\S_2))=&c_1\left((\S_2/\S_1)^*\otimes(\V/\S_2)\right)\\\nonumber
        =&(\rank(\V/\S_2))c_1((\S_2/\S_1)^*)+c_1(\V/\S_2)\\\nonumber
        =&e+4\zeta_\Q.
    \end{align}

    For the class of $\hat\gamma(X)$, we observe that there is a commutative diagram
    \begin{align*}
        \xymatrixcolsep{2pc}
        \xymatrix{
        \P(\T_X)\ar@{^{(}->}[r]\ar[d]_{\cong}
        &\hat\Delta_D\cong\P(\T_{\P^4})\ar[d]^{\hat\pi_\Delta}\\
        X\ar@{^{(}->}[r]^-{\iota_X}\ar@{^{(}->}[ru]^-{\hat\gamma}&\P^4,
        }
    \end{align*}
    where the map on the top row is induced by the total derivative $\T_X\to\T_{\P^4}$ of $\iota_X$. Let $\zeta_\T$ be the hyperplane class of $\P(\T_{\P^4})$ with respect to the vector bundle $\T_{\P^4}$. Lemma \ref{projcur} then gives us
    \begin{align*}
        [\hat\gamma(X)]=(\hat\pi_\Delta^*[\iota_X(X)])\zeta_\T^3+(\hat\pi_\Delta^*(c_1(\T_{\P^4})[\iota_X(X)]+{\iota_X}_*(K)))\zeta_\T^2.
    \end{align*}

    Since $e^4$ is the pullback of the class of a point in $\P^4$ and $de^3$ is the pullback of the class of $X$ in $\P^4$, as $d:=\deg_{\P^4}X$, we have $\hat\pi_\Delta^*{\iota_X}_*(K)=(2g-2)e^4$, $\hat\pi_\Delta^*[\iota_X(X)]=de^3$, and $\hat\pi_\Delta^*c_1(\T_{\P^4})=5e$. In addition, since $\T_{\P^4}=\S^*\otimes\Q$, where $\S$ and $\Q$ denote the universal line and quotient bundles of $\P^4$, we have $\O_{\P(\T_{\P^4})}(-1)=\hat\pi_\Delta^*(\S^*)\otimes \O_{\P(\Q)}(-1)$, cf. \cite[Corollary 9.5]{3264}. Thus, $\zeta_\T=\zeta_\Q-e$. Therefore,
    \begin{align}\label{hatX}
        [\hat\gamma(X)]=&de^3(\zeta_\Q-e)^3+((5e)(de^3)+(2g-2)e^4)(\zeta_\Q-e)^2\\\nonumber
        =&de^3\zeta_\Q^3+d^\vee e^4\zeta_\Q^2.
    \end{align}    

    To finish the calculation, we need the structure of the Chow ring $\A(\tilde\Delta_D)$.
    \begin{prop}\label{blowupl}
        Regarding $\tilde\Delta_D$ as the projectification of the vector bundle $\Hom(\S_2/\S_1,\V/\S_2)$ over $\hat\Delta_D\cong\P(\Q)$, then Chow ring of $\tilde\Delta_D$ is given by
        $$\A(\tilde\Delta_D)=\frac{\Z[e,\zeta_\Q,\zeta_\H]}{\left(e^5,\sum_{i=0}^4e^i\zeta_\Q^{4-i},\sum_{j=0}^3c_j(\Hom(\S_2/\S_1,\V/\S_2))\zeta_\H^{3-j}\right)},$$
        where $e$ is induced by the hyperplane class of $\P^4$, $\zeta_\Q$ the hyperplane section class of $\P(\Q)$, and $\zeta_\H$ the hyperplane section class of $\P(\Hom(\S_2/\S_1,\V/\S_2))$.
    \end{prop}
    In particular, we note that $e^4\zeta_\Q^4=0$, $e^3\zeta_\Q^4=-e^4\zeta_\Q^3$, $\zeta_\H^3=-c_1(\Hom(\S_2/\S_1,\V/\S_2))=-(e+4\zeta_\Q)\zeta_\H^2$, and the class of a point is $e^4\zeta_\Q^3\zeta_\H^2$.

    \medskip

    Now, put (\ref{Hch}) and $(\ref{hatX})$ back into (\ref{tildeX}), we obtain
    \begin{align*}
        [\tilde\Delta_\Gamma]=&(de^3\zeta_\Q^3+d^\vee e^4\zeta_\Q^2)\zeta_\H^2+\left((e+4\zeta_\Q)(de^3\zeta_\Q^3+d^\vee e^4\zeta_\Q^2)+(2g-2)e^4\zeta_\Q^3\right)\zeta_\H\\
        =&de^3\zeta_\Q^3\zeta_\H^2+d^\vee e^4\zeta_\Q^2\zeta_\H^2+de^4\zeta_\Q^3\zeta_\H+4de^3\zeta_\Q^4\zeta_\H+4d^\vee e^4\zeta_\Q^3\zeta_\H+(2g-2)e^4\zeta_\Q^3\zeta_\H\\
        =&de^3\zeta_\Q^3\zeta_\H^2+d^\vee e^4\zeta_\Q^2\zeta_\H^2+(5d+10g-10)e^4\zeta_\Q^3\zeta_\H
    \end{align*}
    as a class in $\A(\P(\tilde\Delta_D))$, and so
    \begin{align*}
        [\tilde\Delta_\Gamma]={j_D}_*\left(de^3\zeta_\Q^3\zeta_\H^2+d^\vee e^4\zeta_\Q^2\zeta_\H^2+(5d+10g-10)e^4\zeta_\Q^3\zeta_\H\right)\in\A(\tilde D_1)
    \end{align*}

    Therefore, by the formula above, formula (\ref{TtilD}), Proposition \ref{blowupp}, and Proposition \ref{blowupl}, we arrive at
    \begin{align*}
        &i^\dagger_{D*}i_D^{\dagger*}(c_1(\T_{\tilde D_1}))\\
        =&[\tilde\Delta_\Gamma]c_1(\T_{\tilde D_1})\\
        =&{j_D}_*\left(de^3\zeta_\Q^3\zeta_\H^2+d^\vee e^4\zeta_\Q^2\zeta_\H^2+(5d+10g-10)e^4\zeta_\Q^3\zeta_\H\right)\left(\pi_D^*(4\zeta_1+4\zeta_2+7e)-{j_D}_*(2)\right)\\
        =&{j_D}_*\left(\left(de^3\zeta_\Q^3\zeta_\H^2+d^\vee e^4\zeta_\Q^2\zeta_\H^2+(5d+10g-10)e^4\zeta_\Q^3\zeta_\H\right)(8\zeta_\Q+7e+2\zeta_\H)\right)\\
        =&{j_D}_*\left(-de^4\zeta_\Q^3\zeta_\H^2+8d^\vee e^4\zeta_\Q^3\zeta_\H^2+2(5d+10g-10)e^4\zeta_\Q^3\zeta_\H^2+2(de^3\zeta_\Q^3+d^\vee e^4\zeta_\Q^2)\zeta_\H^3\right)\\
        =&{j_D}_*\left((8d^\vee+9d+20g-20)e^4\zeta_\Q^3\zeta_\H^2+2(de^3\zeta_\Q^3+d^\vee e^4\zeta_\Q^2)(-e-4\zeta_\Q)\zeta_\H^2\right)\\
        =&{j_D}_*\left((15d+20g-20)e^4\zeta_\Q^3\zeta_\H^2\right),
    \end{align*}
    where $i_D^\dagger:\tilde\Delta_\Gamma\hookto\tilde D_1$ is the inclusion and $e^4\zeta_\Q^3\zeta_\H^2$ is the class of a point in $\tilde D_1$. So
    \begin{align}\label{tantilD}
        \ll i_D^{\dagger*}c_1(\T_{\tilde D_1})\rr=(15d+20g-20)P,
    \end{align}
    where $P\in\NN(\tilde\Delta_\Gamma)$ stands for the class of a point.

    \medskip

    To compute $\ll D_1^\dagger\rr$, we also need the class $\ll i^{\dagger*}c_1(\T_B)\rr$. By applying Lemma \ref{blowupch} to diagram (\ref{blowup}), we have
    
    \begin{align*}
        c_1(\T_B)=&\pi^*c_1(\T_{G\times G})+j_*(1-\codim_{G\times G}\Delta_G)\\\nonumber
        =&5\pi^*(\sigma_1\otimes1+1\otimes\sigma_1)-j_*(5).
    \end{align*}

    Meanwhile, from (\ref{gamgeo4}) we have
    \begin{align*}
        [\tilde\Delta_\Gamma]=d^\vee j_*(\bar\sigma_{32}\zeta^5)+(5d^\vee+2g-2)j_*(\bar\sigma_{33}\zeta^4)\in\A(B).
    \end{align*}
    Thus
    \begin{align*}
        i^\dagger_*i^{\dagger*}c_1(\T_B)=&[\tilde\Delta_\Gamma]c_1(\T_B)\\
        =&j_*\left((d^\vee\bar\sigma_{32}\zeta^5+(5d^\vee+2g-2)\bar\sigma_{33}\zeta^4)(10\bar\sigma_1+5\zeta)\right)\\
        =&j_*\left(10d^\vee\bar\sigma_{33}\zeta^5+5d^\vee\bar\sigma_{32}\zeta^6+5(5d^\vee+2g-2)\bar\sigma_{33}\zeta^5\right)\\
        =&j_*\left((10d^\vee+10g-10)\bar\sigma_{33}\zeta^5\right).
    \end{align*}

    Hence,
    \begin{align}\label{chtanB}
        \ll i^{\dagger*}c_1(\T_B)\rr=(10d^\vee+10g-10)P.
    \end{align}

    \medskip

    Now, we know that $[D_1^\dagger]=\pi^{\dagger*}[\tilde D_1]-j^\dagger_*(\mu)$ for some class $\mu\in\A(E^\dagger)$. To solve $\mu$, we compare the coefficients in $[E^\dagger][D_1^\dagger]=[E^\dagger\cap D_1^\dagger]$. On the one hand, we have
    \begin{align*}
        [E^\dagger][D_1^\dagger]=&j^\dagger_*(1)\left(\pi^{\dagger*}[\tilde D_1]-j^\dagger_*(\mu)\right)=j^\dagger_*\left(\pi^{\dagger*}_Ei^{\dagger*}[\tilde D_1]+\mu\zeta^\dagger\right)=j^\dagger_*\left(\mu\zeta^\dagger\right)
    \end{align*}
    as $[\tilde D_1]$ is of codimension $2$, where $\zeta^\dagger$ is the hyperplane section class of the projective bundle $E^\dagger\cong\P(\N_{\tilde\Delta_\Gamma,B})$.

    On the other hand, $[E^\dagger\cap D_1^\dagger]\in\A(E^\dagger)$ can be treated as the class of the subbundle $\P(\N_{\tilde\Delta_\Gamma,\tilde D_1})$ in $E^\dagger\cong\P(\N_{\tilde\Delta_\Gamma,B})$. Thus, thanks to Lemma \ref{subproj}, we see that
    \begin{align*}
        [E^\dagger\cap D_1^\dagger]=&j^\dagger_*\left((\zeta^\dagger)^2+\pi^{\dagger*}_Ec_1(\N_{\tilde\Delta_\Gamma,B}/\N_{\tilde\Delta_\Gamma,\tilde D_1})\zeta^\dagger\right)\\\nonumber
        =&j^\dagger_*\left((\zeta^\dagger)^2+\pi^{\dagger*}_E(i^{\dagger*}c_1(\T_B)-i_D^{\dagger*}c_1(\T_{\tilde D_1}))\zeta^\dagger\right).
    \end{align*}

    Hence $\mu=\zeta^\dagger+\pi_E^{\dagger*}\left(i^{\dagger*}c_1(\T_B)-i_D^{\dagger*}c_1(\T_{\tilde D_1})\right)$ and
    \begin{align*}
        [D_1^\dagger]=&\pi^{\dagger*}\left(\pi^*(\sigma_2\otimes\sigma_0+\sigma_1\otimes\sigma_1+\sigma_0\otimes\sigma_2)-j_*(5\bar\sigma_1+3\zeta)\right)\\\nonumber
        &-j^\dagger_*\left(\zeta^\dagger+\pi_E^{\dagger*}\left(i^{\dagger*}c_1(\T_B)-i_D^{\dagger*}c_1(\T_{\tilde D_1})\right)\right).
    \end{align*}

    In particular, if we write the class of a fibre in $\NN(E^\dagger)$ by $F:=\pi_E^{\dagger*}(P)$, then
    \begin{align*}
        \ll D_1^\dagger\rr=&\ll\pi^{\dagger*}\left(\pi^*(\sigma_2\otimes\sigma_0+\sigma_1\otimes\sigma_1+\sigma_0\otimes\sigma_2)-j_*(5\bar\sigma_1+3\zeta)\right)\rr\\\nonumber
        &-j^\dagger_*\left(\ll\zeta^\dagger\rr+\pi_E^{\dagger*}\left(\ll i^{\dagger*}c_1(\T_B)\rr-\ll i_D^{\dagger*}*c_1(\T_{\tilde D_1})\rr\right)\right)\\\nonumber
        =&\ll\pi^{\dagger*}\left(\pi^*(\sigma_2\otimes\sigma_0+\sigma_1\otimes\sigma_1+\sigma_0\otimes\sigma_2)-j_*(5\bar\sigma_1+3\zeta)\right)\rr\\\nonumber
        &-j^\dagger_*\left(\ll\zeta^\dagger\rr+(10d^\vee-15d-10g+10)F\right).
    \end{align*}

    \subsection{The class $\ll\Gamma^\dagger\rr$ and the product $\ll D_1^\dagger\rr\ll\Gamma^\dagger\rr$}\hfill\\
    Now we compute $\ll\Gamma^\dagger\rr$. To begin, we know that $[\Gamma^\dagger]=\pi^{\dagger*}[\tilde\Gamma]-j^\dagger_*(\nu)$ for some $\nu\in \A(E^\dagger)$. On the one hand, we have
    \begin{align*}
        [E^\dagger][\Gamma^\dagger]=j^\dagger_*(1)\left(\pi^{\dagger*}[\tilde\Gamma]-j^\dagger_*(\nu)\right)
        =j^\dagger_*\left(\pi_E^{\dagger*}i^{\dagger*}[\tilde\Gamma]+\nu\zeta^\dagger\right)=j^\dagger_*\left(\nu\zeta^\dagger\right).
    \end{align*}

    On the other hand, we see that $E^\dagger\cap\Gamma^\dagger\cong\P(\N_{\tilde\Delta_\Gamma/\tilde\Gamma})$ is a projective subbundle of $E^\dagger\cong\P(\N_{\tilde\Delta_\Gamma/B})$. Thus Lemma \ref{subproj} yields
    \begin{align}\label{ddelta}
        j^\dagger_*[E^\dagger\cap\Gamma^\dagger]=&j^\dagger_*\left((\zeta^{\dagger})^{10}+\pi_E^{\dagger*}c_1\left(\N_{\tilde\Delta_\Gamma,B}/\N_{\tilde\Delta_\Gamma,\tilde\Gamma}\right)(\zeta^{\dagger})^9\right)\\\nonumber
        =&j^\dagger_*\left((\zeta^{\dagger})^{10}+\pi_E^{\dagger*}\left(i^{\dagger*}c_1(\T_B)-\iota_{\tilde\Delta_\Gamma\tilde\Gamma}^*c_1(\T_{\tilde\Gamma})\right)(\zeta^{\dagger})^9\right).
    \end{align}

    Hence $\nu=(\zeta^\dagger)^9+\pi_E^{\dagger*}\left(i^{\dagger*}c_1(\T_B)-\iota_{\tilde\Delta_\Gamma\tilde\Gamma}^*c_1(\T_{\tilde\Gamma})\right)(\zeta^{\dagger})^8$ and, thanks to Lemma \ref{tilgam},
    \begin{align*}
        [\Gamma^\dagger]=&\pi^{\dagger*}\left((d^\vee)^2\pi^*(\sigma_{32}\otimes\sigma_{32})-d^\vee j_*\left(\bar\sigma_{32}\zeta^4\right)-(5d^\vee+2g-2)j_*(\bar\sigma_{33}\zeta^3)\right)\\
        &-j^\dagger_*\left((\zeta^\dagger)^9+\pi_E^{\dagger*}\left(i^{\dagger*}c_1(\T_B)-\iota_{\tilde\Delta_\Gamma\tilde\Gamma}^*c_1(\T_{\tilde\Gamma})\right)(\zeta^{\dagger})^8\right)
    \end{align*}

    Since $\tilde\Delta_\Gamma$ is the diagonal of $\tilde\Gamma\cong\Gamma\cong X\times X$, we have $\ll\iota_{\tilde\Delta_\Gamma\tilde\Gamma}^*c_1(\T_{\tilde\Gamma})\rr=(4-4g)P$, where $P$ is the class of a point. Hence, together with (\ref{chtanB}) we arrive at
    \begin{align*}
        \ll\Gamma^\dagger\rr=&\ll\pi^{\dagger*}\left((d^\vee)^2\pi^*(\sigma_{32}\otimes\sigma_{32})-d^\vee j_*\left(\bar\sigma_{32}\zeta^4\right)-(5d^\vee+2g-2)j_*(\bar\sigma_{33}\zeta^3)\right)\rr\\\nonumber
        &-j^\dagger_*\left(\ll\zeta^\dagger\rr^9+\pi_E^{\dagger*}\left(\ll i^{\dagger*}c_1(\T_B)\rr-\ll\iota_{\tilde\Delta_\Gamma\tilde\Gamma}^*c_1(\T_{\tilde\Gamma})\rr\right)\ll\zeta^{\dagger}\rr^8\right)\\\nonumber
        =&\ll\pi^{\dagger*}\left((d^\vee)^2\pi^*(\sigma_{32}\otimes\sigma_{32})-d^\vee j_*\left(\bar\sigma_{32}\zeta^4\right)-(5d^\vee+2g-2)j_*(\bar\sigma_{33}\zeta^3)\right)\rr\\\nonumber
        &-j^\dagger_*\left(\ll\zeta^\dagger\rr^9+(10d^\vee+14g-14)F\ll\zeta^{\dagger}\rr^8\right).
    \end{align*}

    Thus
    \begin{align*}
        \ll D_1^\dagger\rr\ll\Gamma^\dagger\rr=&\ll\left(\pi^{\dagger*}[\tilde D_1]-j^\dagger_*(\mu)\right)\left(\pi^{\dagger*}[\tilde\Gamma]-j^\dagger_*(\nu)\right)\rr\\
        =&\pi^{\dagger*}\ll[\tilde D_1][\tilde\Gamma]\rr-j^\dagger_*\ll \mu\rr\ll\nu\rr\ll\zeta^\dagger\rr\\
        =&\pi^{\dagger*}\ll(d^\vee)^2\pi^*(\sigma_{32}\otimes\sigma_{32})- j_*\left(d^\vee\bar\sigma_{32}\zeta^4+(5d^\vee+2g-2)\bar\sigma_{33}\zeta^3\right)\rr\\
        &\qquad\ll\left(\pi^*(\sigma_2\otimes\sigma_0+\sigma_1\otimes\sigma_1+\sigma_0\otimes\sigma_2)-j_*(5\bar\sigma_1+3\zeta)\right)\rr\\
        &-j^\dagger_*\left(\ll\zeta^\dagger\rr+(10d^\vee-15d-10g+10)F\right)\\
        &\qquad\left(\ll\zeta^\dagger\rr^9+(10d^\vee+14g-14)F\ll\zeta^{\dagger}\rr^8\right)\ll\zeta^\dagger\rr\\
        =&(d^\vee)^2\pi^{\dagger*}\pi^*\ll\sigma_{33}\otimes\sigma_{33}\rr\\
        &\qquad-\pi^{\dagger*}j_*\ll5d^\vee\bar\sigma_{33}\zeta^5+3d^\vee\bar\sigma_{32}\zeta^6+3(5d^\vee+2g-2)\bar\sigma_{33}\zeta^5\rr\\
        &-j^\dagger_*\ll\zeta^\dagger\rr^{11}-(20d^\vee-15d+4g-4)j^\dagger_*F\ll\zeta^\dagger\rr^{10}.
    \end{align*}

    Thanks to (\ref{chtanB}) and $\tilde\Delta_\Gamma\cong X$, we have
    \begin{align*}
        \ll c_1(\N_{\tilde\Delta_\Gamma,B})\rr=\ll i^{\dagger*}c_1(\T_B)-c_1(\T_{\tilde\Delta_\Gamma})\rr=(10d^\vee+12g-12)P.
    \end{align*}
    Together with $c_1(\T_G)=5\sigma_1$, we see that $\zeta^6=-5\bar\sigma_1\zeta^5$ and $\ll\zeta^\dagger\rr^{11}=-(10d^\vee+12g-12)F\ll\zeta^\dagger\rr^{10}$. Also observe that $\pi^{\dagger*}\pi^*\ll\sigma_{33}\otimes\sigma_{33}\rr=\pi^{\dagger*}j_*\ll\bar\sigma_{33}\zeta^5\rr=j^\dagger_*F\ll\zeta^\dagger\rr^{10}$ is the class of a point in $\N(B^\dagger)$, then the computation above becomes
    \begin{align}\label{daggerint}
        \ll D_1^\dagger\rr\ll\Gamma^\dagger\rr=&(d^\vee)^2P-\pi^{\dagger*}j_*\ll5d^\vee\bar\sigma_{33}\zeta^5-15d^\vee\bar\sigma_{33}\zeta^5+(15d^\vee+6g-6)\bar\sigma_{33}\zeta^5\rr\\\nonumber
        &+(10d^\vee+12g-12)j^\dagger_*F\ll\zeta^\dagger\rr^{10}-(20d^\vee-15d+4g-4)j^\dagger_*F\ll\zeta^\dagger\rr^{10}\\\nonumber
        =&\left((d^\vee)^2-15d^\vee+15d+2g-2\right)P.
    \end{align}

    \subsection{The class $\ll\epsilon_{\Delta^\dagger_\Gamma}\rr$}\hfill\\
    In this section, we compute the class $\ll\epsilon_{\Delta^\dagger_\Gamma}\rr$. To begin, we first observe that, by Lemma \ref{theC}, $\Delta^\dagger_\Gamma$ is the diagonal of $\Gamma^\dagger\cong\tilde\Gamma\cong\Gamma:=\gamma(X)\times\gamma(X)\cong X\times X$, so $\Delta^\dagger_\Gamma\cong X$ and $\iota_{\Delta^\dagger_\Gamma\Gamma^\dagger}:\Delta^\dagger_\Gamma\cong X\hookto\Gamma^\dagger\cong X\times X$ is the inclusion of the diagonal. Thus, by identifying $\A(\Delta^\dagger_\Gamma),\A(\Gamma^\dagger)$ and $\A(X),\A(X\times X)\cong\A(X)\otimes\A(X)$, writing the canonical class of $X$ by $K:=c_1(\T^*_X)$, and recalling that $\deg K=2g-2$ with $g$ the genus of $X$, we immediately obtain that
    \begin{align}\label{tancur}
        \ll c_1(\T_{\Delta^\dagger_\Gamma})\rr=(2-2g)P
    \end{align}
    and
    \begin{align}\label{tangam}
        \ll\iota_{\Delta^\dagger_\Gamma\Gamma^\dagger}^*[\Gamma^\dagger]\rr=\ll[\Delta_X](-K\otimes1-1\otimes K)\rr=4-4g.
    \end{align}

\medskip

    Meanwhile, by applying Lemma \ref{blowupch} to diagram (\ref{BlowUp}), we obtain
    \begin{align*}
        c_1(\T_{B^\dagger})=&\pi^{\dagger*}c_1(\T_B)+j^\dagger_*(1-\codim_B\tilde\Delta_\Gamma)\\
        =&\pi^{\dagger*}\left(5\pi^*(\sigma_1\otimes1+1\otimes\sigma_1)-j_*(5)\right)-j^\dagger_*(10),
    \end{align*}
    and from (\ref{ddelta}) we have
    \begin{align*}
        \ll\Delta_\Gamma^\dagger\rr=&j^\dagger_*\ll E^\dagger\cap\Gamma^\dagger\rr\\
        =&j^\dagger_*\left(\ll\zeta^{\dagger}\rr^{10}+\pi_E^{\dagger*}\ll i^{\dagger*}c_1(\T_B)-\iota_{\tilde\Delta_\Gamma\tilde\Gamma}^*c_1(\T_{\tilde\Gamma})\rr\ll\zeta^{\dagger}\rr^9\right)\\
        =&j^\dagger_*\left(\ll\zeta^{\dagger}\rr^{10}+(10d^\vee+14g-14)F\ll\zeta^{\dagger}\rr^9\right).
    \end{align*}

    Thus, with (\ref{chtanB}),  we obtain
    \begin{align*}
        \iota_{\Delta^\dagger_\Gamma*}\iota_{\Delta^\dagger_\Gamma}^*\ll c_1(\T_{B^\dagger})\rr=&\ll\Delta_\Gamma^\dagger\rr\ll c_1(\T_{B^\dagger})\rr\\
        =&j^\dagger_*\left(\ll\zeta^{\dagger}\rr^{10}+(10d^\vee+14g-14)F\ll\zeta^{\dagger}\rr^9\right)\ll(\pi_E^{\dagger*}i^{\dagger*}c_1(\T_B)+10\zeta^\dagger\rr\\
        =&j^\dagger_*\left(\ll\zeta^{\dagger}\rr^{10}+(10d^\vee+14g-14)F\ll\zeta^{\dagger}\rr^9\right)\left((10d^\vee+10g-10)F+10\ll\zeta^\dagger\rr\right)\\
        =&j^\dagger_*\left(10\ll\zeta^{\dagger}\rr^{11}+(110d^\vee+150g-150)F\ll\zeta^{\dagger}\rr^{10}\right)\\
        =&(10d^\vee+30g-30)j^\dagger_*F\ll\zeta^{\dagger}\rr^{10},
    \end{align*}
    and so 
    \begin{align}\label{tanB}
        \iota_{\Delta^\dagger_\Gamma}^*\ll c_1(\T_{B^\dagger})\rr=(10d^\vee+30g-30)P,
    \end{align}
    where $\iota_{\Delta^\dagger_\Gamma}:\Delta^\dagger_\Gamma\hookto B^\dagger$ is the inclusion map.

\medskip

    For the term $\ll\iota_{\Delta^\dagger_\Gamma D_1^\dagger}^*c_1(D_1^\dagger)\rr$, we consider the blowup diagram
    \begin{align}\label{ddblowup}
            \xymatrixcolsep{2pc}
            \xymatrix{
            E_D^\dagger\ar@{^{(}->}[r]^-{j_D^\dagger}\ar[d]_{\pi_{E_D}^\dagger}
            &D_1^\dagger\cong\Bl_{\tilde\Delta_\Gamma}\tilde D_1\ar[d]^{\pi_D^\dagger}\\
            \tilde\Delta_\Gamma\ar@{^{(}->}[r]^{i_D^\dagger}&\tilde D_1.
            }
        \end{align}
    
    Observe that $\Delta^\dagger_\Gamma\cong\P(\N_{\tilde\Delta_\Gamma,\tilde\Gamma})$ can be treated as a projective subbundle of $E_D^\dagger:=\P(\N_{\tilde\Delta_\Gamma,\tilde D_1})$ since, from Lemma \ref{theC}, we know that $\tilde\Gamma$ and $\tilde D_1$ tangent along $\tilde\Delta_\Gamma$. Thus, due to Lemma \ref{subproj}, we see that 
    \begin{align*}
        [\Delta^\dagger_\Gamma]=&(\zeta^\dagger_D)^8+\pi_D^{\dagger*}c_1(\N_{\tilde\Delta_\Gamma,\tilde D_1}/\N_{\tilde\Delta_\Gamma,\tilde\Gamma})(\zeta^\dagger_D)^7\\
        =&(\zeta^\dagger_D)^8+\pi_D^{\dagger*}\left(i_D^{\dagger*}c_1(\T_{\tilde D_1})-\iota_{\tilde\Delta_\Gamma\tilde\Gamma}^*c_1(\T_{\tilde\Gamma})\right)(\zeta^\dagger_D)^7
    \end{align*}
    as a class in $\A(E_D^\dagger)$, where $\zeta^\dagger_D$ is the hyperplane section class of the projective bundle $E_D^\dagger$. From (\ref{tantilD}) and the fact that $\ll\iota_{\tilde\Delta_\Gamma\tilde\Gamma}^*c_1(\T_{\tilde\Gamma})\rr=(4-4g)P$ we then obtain
    \begin{align*}
        \ll\Delta^\dagger_\Gamma\rr=j^\dagger_{D*}\ll\zeta^\dagger_D\rr^8+(15d+24g-24)j^\dagger_{D*}F_D\ll\zeta^\dagger_D\rr^7\in\NN(D_1^\dagger),
    \end{align*}
    where $F_D$ is the class of a fibre of the bundle $E_D^\dagger$.

    On the other hand, by applying Lemma \ref{blowupch} to the diagram (\ref{ddblowup}), we deduce that
    \begin{align*}
        c_1(\T_{D_1^\dagger})=\pi_D^{\dagger*}c_1(\T_{\tilde D_1})-j^\dagger_{D*}(8).
    \end{align*}

    Thus, with (\ref{tantilD}),
    \begin{align*}
        \iota_{\Delta^\dagger_\Gamma D_1^\dagger*}\iota_{\Delta^\dagger_\Gamma D_1^\dagger}^*\ll c_1(\T_{D_1^\dagger})\rr=&\ll\Delta^\dagger_\Gamma\rr\ll c_1(\T_{D_1^\dagger})\rr\\
        =&j^\dagger_{D*}\left(\ll\zeta^\dagger_D\rr^8+(15d+24g-24)F_D\ll\zeta^\dagger_D\rr^7\right)\ll\pi_D^{\dagger*}c_1(\T_{\tilde D_1})-j^\dagger_{D*}(8)\rr\\
        =&j^\dagger_{D*}\left(\ll\zeta^\dagger_D\rr^8+(15d+24g-24)F_D\ll\zeta^\dagger_D\rr^7\right)\\
        &\qquad\left(\pi_{E_D}^{\dagger*}i_D^{\dagger*}c_1(\T_{\tilde D_1})+8\ll\zeta_D^\dagger\rr\right)\\
        =&j^\dagger_{D*}\left(\ll\zeta^\dagger_D\rr^8+(15d+24g-24)F_D\ll\zeta^\dagger_D\rr^7\right)\\
        &\qquad\left((15d+20g-20)F_D+8\ll\zeta_D^\dagger\rr\right)\\
        =&j^\dagger_{D*}\left(8\ll\zeta_D^\dagger\rr^9+(135d+212g-212)F_D\ll\zeta^\dagger_D\rr^8\right).
    \end{align*}
    Since $\ll c_1(\N_{\tilde\Delta_\Gamma,\tilde D_1})\rr=\ll i_D^{\dagger*}c_1(\T_{\tilde D_1})-c_1(\T_{\tilde\Delta_\Gamma})\rr=(15d+22g-22)P$, we have $\ll\zeta_D^\dagger\rr^9=-(15d+22g-22)F_D\ll\zeta_D^\dagger\rr^8$, so the computation above yields
    \begin{align}\label{tanD}
        \iota_{\Delta^\dagger_\Gamma D_1^\dagger}^*\ll c_1(\T_{D_1^\dagger})\rr=(15d+36g-36)P.
    \end{align}

\medskip

    Therefore, by (\ref{tancur}), (\ref{tangam}), (\ref{tanB}), and (\ref{tanD}) we conclude that the numerical class of the excess intersection on the diagonal is
    \begin{align}\label{excterm}
        \ll\epsilon_{\Delta^\dagger_\Gamma}\rr=&\ll\iota_{\Delta^\dagger_\Gamma}^*c_1(\T_{B^\dagger})\rr-\ll\iota_{\Delta^\dagger_\Gamma\Gamma^\dagger}^*c_1(\T_{\Gamma^\dagger})\rr-\ll\iota_{\Delta^\dagger_\Gamma D^\dagger_1}^*c_1(\T_{D^\dagger_1})\rr+\ll c_1(\T_{\Delta^\dagger_\Gamma})\rr\\\nonumber
        =&(10d^\vee+30g-30)P-(4-4g)P-(15d+36g-36)P+(2-2g)P\\\nonumber
        =&(10d^\vee-15d-4g+4)P
    \end{align}

\subsection{The skew curves in $\P^4$}\hfill\\

    Finally, we return to the excess intersection formula and deduce the following Lemma.
    \begin{lem}\label{length}
        Let $X\subset\P^4$ be a smooth degree $d$ genus $g$ curve with its Gauss map a diffeomorphism to its image. Then the number of nonskew pairs of tangent lines of $X$, counting possibly degenerate third osculating planes, is
        $$(2d+2g-2)^2-20d-44g+44.$$
    \end{lem}

    \begin{proof}
        From the argument in Section \ref{ExcInt}, formulae (\ref{daggerint}) and (\ref{excterm}), the desired number is given by
        \begin{align*}
            \ll D_1^\dagger\rr\ll\Gamma^\dagger\rr-\ll(\iota_{\Delta^\dagger_\Gamma})_*\epsilon_{\Delta^\dagger_\Gamma}\rr
            =&\left((d^\vee)^2-15d^\vee+15d+2g-2\right)P-(10d^\vee-15d-4g+4)P\\
            =&((d^\vee)^2-10d^\vee-24g+24)P,
        \end{align*}
        where $d^\vee=2d+2g-2$.
    \end{proof}

    \begin{eg}
        Let $X\subset\P^4$ be a canonical genus $5$ curve given by intersecting three generic quadratic hyper surfaces. Then we have $g=5$ and $d=8$ and so by the above lemma the number of nonskew pairs of tangent lines is
        $$(2d+2g-2)^2-20d-44g+44=240.$$
        We have verified this number with Macaulay 2. Note that this number, curiously, coincides with the cardinality of a Steiner complex of a genus $5$ curve, see \cite[Proposition 5.4.7]{CAG}. The relation between theta characteristics and tangents of canonical curves has been observed in \cite{TerrCurve} as well as \cite{Ciro}, but the specific relation to the Steiner complex remains unknown to the author.
    \end{eg}

    \begin{thm}
        The only algebraically skew curves, up to linear equivalence, in $\P^4$ are the rational normal curve and the elliptic normal curves.
    \end{thm}

    \begin{proof}
        A degree $d$ genus $g$ curve $X$ in $\P^4$ is algebraically skew if and only if it has no pairs of intersecting tangent lines. Thus, by the lemma above we have $(d^\vee)^2-10d^\vee-24g+24=0$,
        which implies
        \begin{align*}
            g=\frac{1}{24}((d^\vee)^2-10d^\vee)+1.
        \end{align*}

        Thus from $d^\vee:=2d+2g-2$ we obtain
        \begin{align*}
            d=-g+\frac{1}{2}d^\vee+1=-\frac{1}{24}(d^\vee)^2+\frac{22}{24}d^\vee
        \end{align*}

        By the nondegeneracy of $X$, we have $d\ge4$. Together with the quadratic equation above, we see that $6\le d^\vee\le16$ and $d\le5$. Furthermore, since both $d$ and $d^\vee$ are integers, the remaining options are $(g,d)=(0,4),(1,5),(2,5),(5,4)$. However, since these are nondegenerate smooth curves in $\P^4$, the cases $(5,2),(4,5)$ are in fact not possible, see \cite[p.253]{grif_harr}. Therefore, the remaining possible cases are the rational normal curve with $d=4$ and $g=0$, and the elliptic normal curves with $d=5$ and $g=1$.

        From Lemma \ref{racurve}, we already know that the rational normal curve is skew, so it remains to prove the following lemma.
    
    \begin{lem}
        The smooth elliptic normal curves in $\P^4$ are algebraically skew.
    \end{lem}

    \begin{proof}
        Let $X\hookto\P^4$ be any smooth elliptic normal curve. Suppose for contradiction that $X$ is not skew, then there exist $x\ne y\in X$ such that $T_x\cap T_y\ne\emptyset$. Then, for any other point $z\in X$ with $x\ne x,y$, $T_x,T_y,z$ span a $3$-space in $\P^4$. In other words, the divisor $2x+2y+z$ corresponds to the hyperplane bundle on $X$; that is, $|2x+2y+z|=\O_X(1)$. Let $z_1\ne z_2\in X$, then we would have $|2x+2y+z_1|=\O_X(1)=|2x+2y+z_2|$, which implies that $\O_X(1)\otimes\L=\O_X(1)$, where $\L:=|z_1-z_2|$. This implies that any two generic points $z_1,z_2$ on $X$ are linearly equivalent, which is a contradiction.

        A smooth elliptic normal curve is also algebraically skew. To see this, we again suppose for contradiction that there is a elliptic normal curve $X\hookto\P^4$ that is skew but not algebraically skew. Let $x\in X$ be a point with a degenerate third osculating plane $P$. Then, for any other point $z\ne x\in X$, $P$ and $z$ span a $3$-space in $\P^4$. In other words, $|4x+z|=\O_X(1)$. This again implies that $|z_1-z_2|$ is trivial for any two generic points $z_1,z_2\in X$ and again a contradiction.
    \end{proof}       
    \end{proof}

    \begin{rmk}
        Note that by taking the affine part of an elliptic normal curve $X\subset\P^4=\CP^4$ with the infinity chosen to be a hyperplane intersecting $X$ at one point with multiplicity $5$, we obtain a skew embedding of a punctured torus into $\R^8$, which gives us a nontrivial real totally skew surface in an Euclidean space of dimension $<9$.
    \end{rmk}

\bibliography{mybib}{}
\bibliographystyle{alphaurl}

\end{document}